\documentclass[]{interact}

\usepackage{epstopdf}
\usepackage[caption=false]{subfig}

\usepackage[numbers,sort&compress]{natbib}
\bibpunct[, ]{[}{]}{,}{n}{,}{,}
\makeatletter
\def\NAT@def@citea{\def\@citea{\NAT@separator}}
\makeatother
\usepackage{anyfontsize}
\theoremstyle{plain}
\newtheorem{theorem}{Theorem}[section]
\newtheorem{lemma}[theorem]{Lemma}
\newtheorem{corollary}[theorem]{Corollary}

\theoremstyle{definition}
\newtheorem{definition}[theorem]{Definition}
\newtheorem{remark}[theorem]{Remark}
\newtheorem{example}[theorem]{Example}

\usepackage{xcolor}

\theoremstyle{remark}

\begin{document}
	

	\title{Unbounded Antilinear Operators on Hilbert Spaces}
  \author{\name {Arup Majumdar \thanks{Arup Majumdar (corresponding author). Email address: arupmajumdar93@gmail.com}} \affil{Institute of Mathematics, Czech Academy of Sciences,
			\v{Z}itn\'a 25, Prague, Czech Republic.}}
  
	\maketitle

	\begin{abstract}
    The paper introduces unbounded antilinear operators on Hilbert spaces and develops their fundamental theory. In particular, we establish a closed range theorem, a polar decomposition theorem, and the convexity of the numerical range for antilinear operators. Furthermore, we present several new results on antilinear normal operators and provide necessary and sufficient conditions for the existence of a minimal antilinear normal extension of an antilinear subnormal operator. We further develop a comprehensive characterization of antilinear block operator matrices with purely antilinear entries, establishing necessary and sufficient criteria for their closability through the framework of Schur and quadratic complements.
\end{abstract}

	\begin{keywords}
		conjugation, antilinear normal operators, antilinear block operator matrices.
	\end{keywords}
      \begin{amscode} Primary 47A12, 47B20; Secondary 47A05.\end{amscode}
    \section{Introduction}
With the rapid development of quantum information and quantum computation theory, the study of the preparation, transformation, and measurement of quantum states has become increasingly significant. Antilinearity plays a fundamental role in physics, originating from the seminal observation of Wigner that time-reversal symmetry in quantum mechanics is implemented by antiunitary operators \cite{bargmann}. Consequently, antilinear operators are indispensable in the analysis of symmetries and transformations of quantum systems.

In the framework of quantum field theory, important symmetries such as charge conjugation (C-symmetry), time-reversal symmetry (T-symmetry), and combined parity-time symmetry (PT-symmetry) are all described by antilinear operators \cite{IonG}. Despite their importance, a systematic development of both bounded and unbounded antilinear operators remains limited. Most existing works focus on special classes of bounded antilinear operators of the form $CT$, where $C$ is a conjugation and $T$ is a bounded linear operator on a Hilbert space. In this direction, several results have been obtained for $C$-normal, $C$-quasinormal, and $C$-symmetric operators \cite{ptak2019c, EKo}.

The study of unbounded $C$-symmetric and $C$-self-adjoint operators has earlier been carried out under the terminology of $J$-symmetric and $J$-self-adjoint operators \cite{IanK}. More recently, \cite{Arlinskii} presents a detailed investigation of unbounded $C$-symmetric and $C$-self-adjoint operators, including necessary and sufficient conditions for their existence via operator matrix representations with off-diagonal structures. We also refer to \cite{Fherbut} for the polar decomposition of antilinear normal operators on finite-dimensional Hilbert spaces and its applications to electrical network theory and two-particle wave operators.

This paper aims to contribute to the systematic development of unbounded antilinear operators on Hilbert spaces. Section 2 constitutes the core of the paper, where we establish several fundamental results, including the closed range theorem, the convexity of the numerical range, and various characterizations of unbounded antilinear normal operators. Furthermore, we investigate the existence of minimal antilinear normal extensions of antilinear subnormal operators, develop the polar decomposition for densely defined closed antilinear operators, and obtain commutativity characterizations associated with such decompositions. Moore--Penrose inverses of densely defined closed antilinear operators are also introduced in this section. Section~3 presents several characterizations of antilinear block operator matrices whose entries are all antilinear operators, together with necessary and sufficient conditions for these matrices to be closable in terms of their Schur and quadratic complements.\\

Throughout the paper, the symbols $H, K, K_{1}$ represent complex Hilbert spaces. Consequently, for an (linear and antilinear) operator $T$, it is necessary to specify the subspace on which $T$ is defined, called the domain of $T$ and denoted by $D(T)$. The null space and range space of $T$ are denoted by $N(T)$ and $R(T)$, respectively. For a subset $W_{1} \subseteq X$, the symbol $W_{1}^{\perp}$ denotes its orthogonal complement, while $W_{1} \oplus W_{2}$ denotes the orthogonal direct sum of the subspaces $W_{1}$ and $W_{2}$ of $X$. Furthermore, $T_{\vert{W}}$ denotes the restriction of $T$ to a subspace $W \subseteq D(T)$.
\begin{definition}\cite{majumdar2024hyers}
		Let $ T$ be an (linear or antilinear) operator from a Hilbert space $H $ with domain $D(T) $ to a Hilbert space $K $. If the graph of $T$ is defined by 
		$$ G(T)=\left\{(x,Tx): x\in D(T)\right\} $$ is closed in $H\times K $, then $T$ is called a closed (linear or antilinear) operator. Equivalently, $T $ is a closed operator if $ \{x_n \}$ in $D(T) $ such that $ x_n\rightarrow x $ and $ Tx_n\rightarrow y$ for some $ x\in  H,y\in   K $,  then $ x\in  D(T) $ and $ Tx=y $. That is, $G(T)$ is  a closed subset of $H\times K$ with respect to the graph norm $\|(x, y)\|_T=(\|x\|^2+\|y\|^2)^{1/2}$. It is easy to show that the graph norm  $\|(x, y)\|_T$ is equivalent to the norm $\|x\|+\|y\|$. 	 We note that,  for any densely defined closed (linear or antilinear) operator $T$,  the closure of $D(T) \cap N(T)^{\perp}$ is $ N(T)^\perp$. 
		We say that $S$ is an extension of $ T $ (denoted by $ T\subset S$) if $ D(T)\subset D(S) $ and $ Sx=Tx $ for all $ x\in D(T)$.
		
		An (linear or antilinear) operator $ T $ is said to be closable if $T$ has a closed extension. It follows that $T$ is closable if and only if the closure $\overline {G(T)}$ of $G(T)$ is a graph of an (linear or antilinear) operator. So, $T$ is closable if and only if $(0,y) \in \overline{G(T)}$ implies $y=0.$	It is also possible for a closable (linear or antilinear) operator to have many closed extensions.  Its minimal closed extension is denoted by $ \overline{T}$. That is, every closed extension of $ T $ is also an extension of $ \overline{T}$.
        \end{definition}
        \begin{definition}\cite{majumdar2024hyers}
		A linear subspace $D$ of $D(T)$ is called a core for $T$ if $D$ is dense in $(D(T), \|.\|_{T})$, that is, for each $x\in D(T)$, there exists a sequence $\{x_{n}\}$ in $D$ such that  $\{x_{n}\}$ and $\{ Tx_{n}\}$ converge to $x$ and $Tx$ respectively. Moreover, if $T$ is closed, a linear subspace $D$ of $D(T)$ is a core for $T$ if and only if $T=\overline{T_{\vert{D}}}$.
	\end{definition}

 \section{Characterization of the unbounded antilinear operators}
 From now on, $T$ is a unbounded antilinear between the specified complex Hilbert spaces. A conjugation on $H,$ denoted by $C,$ is an isometric antilinear involution on $K.$ By isometric antilinear involution, we mean that $C(\alpha x + y)= \overline{\alpha} C(x) + C(y)$ with $C^{2} = I,$ and $\langle Cx, Cy\rangle = \langle y, x \rangle.$ It is evident to see that any antilinear operator $T$ from $H$ to $K$ can be written as $T = CT_{1},$ where $T_{1}= CT$ is a linear operator with domain $D(T_{1}) = D(T).$\\
 Let \(T\) be an antilinear operator from \(H\) into \(K\) such that the domain \(D(T)\) is dense in \(H\). Set
\[
D(T^{\sharp}) = \left\{ y \in K : \text{there exists } u \in H \text{ such that } \overline{\langle Tx, y \rangle} = \langle x, u \rangle \text{ for all } x \in D(T) \right\}.
\]

By the Riesz representation theorem, a vector \(y \in K\) belongs to \(D(T^{\sharp})\) if and only if the map
\[
x \mapsto \overline{\langle Tx, y \rangle}
\]
is a bounded linear functional on \((D(T), \|\cdot\|)\), or equivalently, there exists a constant \(c_y > 0\) such that
\[
|\langle Tx, y \rangle| \leq c_y \|x\| \quad \text{for all } x \in D(T).
\]
Here, the adjoint of $T$ is denoted by $T^{\sharp}$ and $T^{\sharp}(y) = u.$ For the definition of the conjugation $C,$ it follows that $C^{\sharp} = C= C^{-1}.$
 \begin{lemma}\label{lemma 1}
 Let $S$ be a densely defined closed linear operator from $H$ to $K$ and $C$ be a conjugation on $K.$ Then $CS$ is closed and $(CS)^{\sharp} = S^{*}C.$
 \end{lemma}
 \begin{proof}
 Let $(x, y) \in \overline{G(CS)}.$ Then there exist $\{x_{n}\}$ in $D(S)$ such that $\{x_{n}\} \to x$ and $CSx_{n} \to y$ as $n \to \infty.$ So, $\{Sx_{n}\}$ is Cauchy and converges to  $C^{-1}y \in K.$ Since $S$ is closed, then $x \in D(CS)~ \text{and}~CSx= y.$ Thus, $CS$ is closed.\\
 The adjoint of $CS$ exists because the domain of $D(CS)= D(S)$ is dense. For all $h \in D(S)$ and a fixed $g \in D((CS)^{\sharp}),$ we get
 $$\langle (CS)^{\sharp}g, h\rangle = \langle CSh, g \rangle = \overline{\langle Sh, Cg\rangle}.$$ So, $\langle Sh, Cg\rangle$ is bounded for all $h \in D(S).$ It shows that $(CS)^{\sharp} \subset S^{*}C.$ Now we claim the reverse inclusion. Let $z \in D(S^{*}C).$ Then for all $u \in D(CS),$ we obtain $$\langle CSu, z \rangle = \overline{\langle Su, Cz \rangle} = \langle S^{*}Cz, u\rangle.$$ Hence, $z \in D((CS)^{\sharp})$ which implies that $(CS)^{\sharp} = S^{*}C.$ 
 \end{proof}
\begin{remark}
When $F$ is a densely defined (linear or antilinear) operator from $H$ to $K$ and $E$ is a bounded (linear or antilinear) operator from $K$ to $K_{1},$ then the adjoint of $EF$ is equal to the product of the adjoints in reverse order.
\end{remark}
 \begin{lemma}\label{lemma 2}
 Let $T$ be a densely defined anilinear operator from $H$ to $K$. Then $T^{\sharp}$ is also a closed antilinear operator. 
\end{lemma}
\begin{proof}
Let $(y, x) \in \overline{G(T^{\sharp})}.$ Then there exists a sequence $y_{n}$ in $D(T^{\sharp})$ such that $y_{n} \to y$ and $T^{\sharp}y_{n} \to x$ as $n \to \infty.$ For all $z \in D(T),$ we get $$\overline{\langle Tz, y\rangle} = \lim_{n\to \infty} \overline{\langle Tz, y_{n} \rangle} = \lim_{n \to 
\infty}\langle z, T^{\sharp}y_{n}\rangle = \langle z, x\rangle.$$ Then $y \in D(T^{\sharp})$ and $T^{\sharp}y = x.$ Therefore, $T^{\sharp}$ is closed. The antilinear property of $T^{\sharp}$ follows from the definition of the adjoint of $T.$
\end{proof}
\begin{remark}\label{remark 3} When $T$ is a densely defined antilinear operator from $H$ to $K,$ then $N(T^{\sharp}) = R(T)^{\perp}.$ Because $z \in N(T^{\sharp})$ if and only if $\langle Tx, z \rangle = 0,~ \text{for all}~ x \in D(T)$ if and only if $z \in R(T)^{\perp}.$
\end{remark}
\begin{theorem}\label{thm 4}
Let $S$ be a densely defined closed linear operator from $H$ to $K$ and $C$ be a conjugation on $K.$ Then the following statements hold:
\begin{enumerate}
\item $((CS)^{\sharp})^{\sharp} = CS.$
\item $R((CS)^{\sharp})^{\perp} = N(CS).$
\end{enumerate}
\end{theorem}
\begin{proof}
$\it(1)$ Lemma \ref{lemma 1} yields that $(CS)^{\sharp} = S^{*}C.$ Now we claim that $D(S^{*}C)$ is dense in order to show the existence of $((CS)^{\sharp})^{\sharp}.$\\
Let $x \in K.$ Then there exists $z \in K$ such that $Cz =x.$ Since the domain $D(S^{*})$ is dense in $K$ because $S$ is closed. So, there exists a sequence $\{x_{n}\}$ in $D(S^{*})$ such that $x_{n} \to z = C^{2}z.$ Again, we get a sequence $\{z_{n}\}$ in $K$ with $Cz_{n} = x_{n},$ for all $n \in \mathbb{N}.$ It follows that $z_{n} \to x,$ as $n \to \infty,$ where $z_{n} \in D(S^{*}C)$ for all $n \in \mathbb{N}.$ Thus, $D(S^{*}C)$ is densely defined.\\
Let $p \in D((CS)^{\sharp})^{\sharp}.$ For all $u \in D(CS)^{\sharp},$
$f(u)= \langle u, ((CS)^{\sharp})^{\sharp}p\rangle = \langle p, S^{*}Cu\rangle$ is bounded.\\
Again, consider a new functional $g(w) = \langle S^{*}w, p \rangle,$ for all $w \in D(S^{*}).$ Since $R(C) = H,$ we also have $w^{'} \in K$ such that $$|g(w)| = |\langle S^{*}Cw^{'}, p\rangle| =|f(w^{'})| \leq \|Cw^{'}\| \|((CS)^{\sharp})^{\sharp}p\| = \|w\| \|((CS)^{\sharp})^{\sharp}p\|.$$ It means that $p \in D(S).$ So, we obtain $f(u) = \langle u, ((CS)^{\sharp})^{\sharp}p\rangle = \langle Sp, Cu\rangle = \langle u, CSp\rangle.$ Hence, $((CS)^{\sharp})^{\sharp} \subset CS.$\\ Again, consider an element $y\in D(CS).$ Then, for all $k \in D(S^{*}C)$,  the functional $f_{1}(k) = \overline{\langle S^{*}Ck, y\rangle} = \langle k, CSy\rangle$ is bounded. Therefore, $y \in D((CS)^{\sharp})^{\sharp}$ which confirms the required equality $((CS)^{\sharp})^{\sharp} = CS.$\\
$\it(2)$ The relation $R((CS)^{\sharp})^{\perp} = N(CS)$ is verified from that Remark \ref{remark 3} and the identity  $((CS)^{\sharp})^{\sharp} = CS.$
\end{proof}
\begin{remark}\label{remark 5}
When T is a densely defined antilinear closed operator from $H$ to $K,$ then $D(T_{1})$ is also a densely defined closed linear operator from $H$ to $K,$ where $T_{1} =CT,$ $T =CT_{1},$ and $C$ is a conjugation on $K.$ Then the previous Theorem \ref{thm 4} and Remark \ref{remark 3} demonstrate that the relations (i) $(T^{\sharp})^{\sharp} = T,$ (ii) $R(T^{\sharp})^{\perp} = N(T),$ and (iii) $ N(T^{\sharp}) = R(T)^{\perp}.$
\end{remark}
\begin{lemma}\label{lemma 6}
Let $S$ be a densely defined closed linear operator from $H$ to $K$ and $C$ be a conjugation on $K.$ Then the following statements hold:
\begin{enumerate}
\item $(CS)^{\sharp}(CS)$ is a positive self-adjoint operator.
\item $(CS)(CS)^{\sharp}$ is a positive self-adjoint operator.
\end{enumerate}
\end{lemma}
\begin{proof}
$\it(1)$ It follows from Lemma \ref{lemma 1} that $(CS)^{\sharp}(CS) = S^{*}S$ is a positive self-adjoint operator.\\
$\it(2)$ We obtain from Lemma \ref{lemma 1} that $(CS)(CS)^{\sharp} = CSS^{*}C$ is positive. Now, we claim that $D(CSS^{*}C) = D(SS^{*}C)$ is dense in $K.$\\
Let $h \in K.$ Then there exists $k \in K$ such that $Ck =h.$ So, we get a sequence $\{p_{n}\}$ in $D(SS^{*})$ with $p_{n} = Cq_{n}$ for all $n \in \mathbb{N}$ such that $p_{n} \to k$ and $q_{n} = Cp_{n} \to Ck= h,$ as $n \to \infty.$ So, $h \in \overline{D(SS^{*}C)}.$ This shows that $D((CS)(CS)^{\sharp})$ is dense in $K.$\\
Let $y \in D(CSS^{*}C)^{*}.$ For all $x \in D((CS)(CS)^{\sharp}),$ we have 
\begin{align}\label{equ 1}
g(x) = \langle x, (CSS^{*}C)^{*}y\rangle = \langle CSS^{*}Cx, y \rangle = \langle Cy, SS^{*}Cx\rangle
\end{align}
is bounded. For any $z \in D(SS^{*}),$ consider $f(z) = \langle SS^{*}z, Cy\rangle.$ Then we get an element $u \in K$ with $Cu = z$ and $$|f(z)| = |\langle SS^{*}Cu, Cy\rangle| = |\langle u, (CSS^{*}C)^{*}y \rangle| \leq \|z\| \|(CSS^{*}C)^{*}y\|.$$ It confirms that $Cy \in D(SS^{*}).$ By the identity (\ref{equ 1}), we get $$\langle x, (CSS^{*}C)^{*}y\rangle = \langle CSS^{*}Cx, y \rangle = \langle SS^{*}Cy, Cx\rangle = \langle x, CSS^{*}Cy\rangle.$$ Hence, $(CSS^{*}C)^{*} \subset (CSS^{*}C).$ Let $p \in D(CSS^{*}C).$ Then, for all $q \in D(SS^{*}C),$ $$f_{1}(q) = \langle CSS^{*}Cq, p \rangle = \langle Cp, SS^{*}C q\rangle = \langle SS^{*}Cp, Cq\rangle = \langle q, CSS^{*}Cp\rangle$$ is bounded. Now, it follows that $p \in D(CSS^{*}C)^{*}.$ Therefore, $$CS(CS)^{\sharp} = CSS^{*}C = (CSS^{*}C)^{*}$$ is a positive self-adjoint operator.
\end{proof}
\begin{theorem}\label{thm 7}
Let $S$ be a densely defined closed linear operator from $H$ to $K$ and $C$ be a conjugation on $K.$ Then the following statements are equivalent:
\begin{enumerate}
 \item $R(CS)^{\sharp}$ is closed.
\item $R((CS)^{\sharp}(CS)) = R(CS)^{\sharp}.$
\item $R((CS)^{\sharp}(CS))$ is closed.
\item $R(CS)$ is closed.
\end{enumerate}
\end{theorem}
\begin{proof}
$\text{(1) } \Rightarrow \text{ (2)}$ Since $R(CS)^{\sharp}$ is closed. Then $R(CS)^{\sharp}= R(S^{*}C) = R(S^{*})$ is closed. So, $R(S^{*}) = R(S^{*}S) = R((CS)^{\sharp}(CS)).$\\
$\text{(2) } \Rightarrow \text{ (3)}$ It is obvious to show that $\vert S \vert = ((CS)^{\sharp}(CS))^{\frac{1}{2}}.$ By the Douglas Theorem of unbounded operators (Theorem 2\cite{MR0203464}), we get $R(\vert S\vert) = R(S^{*}).$ Again, $R(S^{*}) = R(S^{*}C) = R(CS)^{\sharp}.$ From the given condition, we obtain $$R((CS)^{\sharp}(CS)) = R(CS)^{\sharp} = R(\vert S \vert) = R(((CS)^{\sharp}(CS))^{\frac{1}{2}}).$$ Hence, $R((CS)^{\sharp}(CS))$ is closed.\\
$\text{(3) } \Rightarrow \text{ (1)}$ When $R((CS)^{\sharp}(CS)) = R(S^{*}S)$ is closed, then $$R(CS)^{\sharp} = R(S^{*}C) = R(S^{*})$$ is also closed.\\
$\text{(1) } \Rightarrow \text{ (4)}$ Assume $R(CS)^{\sharp} = R(S^{*})$ is closed. Let $x \in D(CS)= D(S).$ Then, we express $x$ as $x = x_{1} + x_{2},$ where $x_{1} \in R(CS)^{\sharp}$ and $x_{2} \in N(CS) = N(S).$ Again, $CSx =CSx_{1}= (CS)(CS)^{\sharp}z_{1},$ for some $z_{1}\in K$ with $x_{1} = (CS)^{\sharp}z_{1}.$ So, 
\begin{align}\label{equ 2}
R(CS) \subset R((CS)(CS)^{\sharp}).
\end{align}
Now, we claim that $R(CS) \supset R((CS)(CS)^{\sharp})^{\frac{1}{2}}.$ Renaming $((CS)(CS)^{\sharp})^{\frac{1}{2}}$ as $A$ and $CS$ as $B,$ we get $AA = BB^{\sharp}.$ Again, we define a new antilinear operator $E$ as $E(B^{\sharp}x) = Ax$ for all $x \in D(BB^{\sharp}) = D(A^{2}),$ and $Ey = 0$ for all $y \in R(B^{\sharp})^{\perp} = N(S).$ By Lemma \ref{lemma 6}, we have $D(A^{2}) = D(BB^{\sharp})$ is a core of $A$ and from the given condition $R(S^{*}) = R(B^{\sharp})$ is closed. It is evident that $CSS^{*}C = (C\vert S^{*}\vert C)^{2}.$\\ 
Now, we will prove that the positive operator $C\vert S^{*}\vert C$ is self-adjoint. Moreover, $D(C \vert S^{*}\vert C)$ is dense in $K$ because of $D(C \vert S^{*}\vert C) \supset D(CSS^{*}C).$\\
Let $y \in D((C\vert S^{*}\vert C)^{*}).$ Then for all $x \in D(C\vert S^{*}\vert C),$ we have 
\begin{align}\label{equ 3}
\langle x, (C\vert S^{*}\vert C)^{*}y\rangle = \langle C\vert S^{*}\vert Cx, y \rangle = \langle Cy, \vert S^{*}\vert Cx\rangle.
\end{align}
Again, for any $z \in D(\vert S^{*} \vert)$, we obtain $f(z) = \langle \vert S^{*}\vert z, Cy\rangle = \langle \vert S^{*}\vert Cz_{1}, Cy \rangle,$ for some $z_{1} \in K$ with $Cz_{1}= z.$ It follows from the identity (\ref{equ 3}) that $$\vert f(z)\vert \leq \|z_{1}\| \|(C\vert S^{*}\vert C)^{*}y\| = \|z\| \|(C\vert S^{*}\vert C)^{*}y\|.$$ Thus, $Cy \in D(\vert S^{*} \vert).$ From the relation (\ref{equ 3}), we can express $$\langle x, (C\vert S^{*}\vert C)^{*}y\rangle = \langle \vert S^{*}\vert Cy,  Cx\rangle = \langle x, C\vert S^{*}\vert C y \rangle.$$ which implies that $(C\vert S^{*}\vert C)^{*} \subset C \vert S^{*}\vert C.$\\
Let $z \in D(C \vert S^{*}\vert C).$ Then, for all $u \in D(C \vert S^{*}\vert C),$ we have $$\langle C \vert S^{*}\vert C u, z \rangle = \langle Cz, \vert S^{*}\vert Cu \rangle = \langle u, C\vert S^{*}\vert Cz \rangle.$$ So, $z \in D((C\vert S^{*}\vert C)^{*}).$ Now. we ensure that $C\vert S^{*}\vert C$ is self-adjoint. Hence, $A = ((CS)(CS)^{\sharp})^{\frac{1}{2}} = C\vert S^{*}\vert C.$ When $w \in D(B^{\sharp}) = D(S^{*}C),$ then $Cw \in D(\vert S^{*} \vert)$ which leads to $w \in D(\vert S^{*}\vert C) = D(A).$ It yields that there exists a sequence $w_{n}$ in $D(A^{2})$ such that $w_{n} \to w$ and $Aw_{n} \to Aw,$ as $n \to \infty.$ 
Now, we extend the antilinear operator $E$ on $K$, by defining $E(B^{\sharp}w) = Aw,$ we obtain 
\begin{align}\label{equ 4}
\|E(B^{\sharp}w)\| = \|Aw\| = \lim_{n \to \infty} \|Aw_{n}\|= \lim_{n \to \infty}\|B^{\sharp}w_{n}\|.
\end{align}
It means that $\{B^{\sharp}w_{n}\}$ is Cauchy because $\{Aw_{n}\}$ is Cauchy and $B^{\sharp}$ is closed (by Lemma \ref{lemma 2}) which yields that $B^{\sharp}w_{n} \to B^{\sharp}w,$ as $n \to \infty.$ By the identity (\ref{equ 4}), we get $\|E(B^{\sharp}w)\| = \|B^{\sharp}w\|,$ from which it follows that $E$ is bounded antilinear. It is obvious that $D(A) = D(\vert S^{*} \vert C) = D(B).$ The relation $EB^{\sharp} = A$ says that $A = A^{*} = BE^{\sharp}.$\\ It concludes that 
\begin{align}\label{equ 5}
R((CS)(CS)^{\sharp})^{\frac{1}{2}} = R(A) \subset R(B) = R(CS)
\end{align}
By the identities (\ref{equ 2}) and (\ref{equ 5}), we get $$R((CS)(CS)^{\sharp}) \subset R((CS)(CS)^{\sharp})^{\frac{1}{2}} \subset R(CS) \subset R((CS)(CS)^{\sharp}).$$ Therefore, $R(CS) = R((CS)(CS)^{\sharp}) = R((CS)(CS)^{\sharp})^{\frac{1}{2}}$ is closed.\\
$\text{(4) } \Rightarrow \text{ (1)}$ Since $R(CS)$ is closed. Let $q \in R(CS)^{\sharp}.$ Then $q = (CS)^{\sharp}w,$ for some $w = w_{1} + w_{2} \in D(S^{*}C),$ where $w_{1} \in N(S^{*}C),$ and $w_{2} \in N(S^{*}C)^{\perp} \cap D(S^{*}C).$ Again, $q = (CS)^{\sharp}(CS)w_{2}^{'},$ where $w_{2} = (CS)w_{2}^{'}$ for some $w_{2}^{'} \in H.$ Hence $$R(CS)^{\sharp} \subset R((CS)^{\sharp}(CS)) \subset R(CS)^{\sharp}.$$
Therefore, $R(CS)^{\sharp}$ is closed.
\end{proof}
\begin{corollary}\label{cor 8}
Let $T$ be a densely defined closed antilinear operator from $H$ to $K.$ Then $R(T)$ is closed if and only if $R(T^{\sharp})$ is closed.
\end{corollary}
\begin{proof}
Since, $T$ is  densely defined closed. Then $T_{1} = CT$ is also a densely defined closed operator from $H$ to $K,$ where C is a conjugation on $K.$ By Theorem \ref{thm 7}, we obtain that $R(T) = R(CT_{1})$ is closed if and only if $R(T^{\sharp}) = R(CT_{1})^{\sharp}$ is closed.
\end{proof}
\begin{theorem}\label{thm 9}
Let $T$ be a densely defined closed antilinear operator from $H$ to $K.$ If $\dim R(T)^{\perp} < \infty,$ then $R(T)$ is closed.
\end{theorem}
\begin{proof}
Firstly, we write $T = CT_{1},$ where $T_{1} = CT$ is a densely defined closed operator from $H$ to $K$ and $C$ is a conjugation on $K.$ By Theorem \ref{thm 4} and Lemma \ref{lemma 1}, we get $$CT_{1} = ((CT_{1})^{\sharp})^{\sharp} = (T_{1}^{*}C)^{\sharp} = (CT_{0})^{\sharp},$$ where the densely defined closed operator $T_{0} = CT_{1}^{*}C.$ So, $CT_{0} = (CT_{1})^{\sharp}.$ By the given condition, we obtain $$\dim R(T_{0}^{*})^{\perp} = \dim R(T_{0}^{*}C)^{\perp}  = \dim R(CT_{1})^{\perp} < \infty.$$ So, $R(T_{0}^{*})$ is closed. Now, we claim that $T_{0}^{*} = CT_{1}C.$ Let $y \in D(T_{0}^{*}).$ For all $x \in D(T_{0}),$ we have 
\begin{align}\label{equ 6}
\langle x, T_{0}^{*}y\rangle = \langle T_{0}x, y \rangle = \langle Cy, T_{1}^{*}Cx\rangle.
\end{align}
Consider a new functional for  $z \in D(T_{1}^{*}),$ $f(z) = \langle T_{1}^{*}z, Cy \rangle =  \langle T_{1}^{*}Cz_{1}, Cy \rangle,$ where $Cz_{1} = z$ for some $z_{1} \in K.$ From the identity (\ref{equ 6}), it follows that $f(z)$ is bounded for all $z \in D(T_{1}^{*}).$ Thus, $Cy \in D(T_{1}).$ Again, from equation (\ref{equ 6}), we deduce that
$$ \langle x, T_{0}^{*}y\rangle = \langle T_{0}x, y \rangle = \langle T_{1}Cy, Cx\rangle = \langle x, CT_{1}Cy\rangle.$$ Hence, $T_{0}^{*} \subset CT_{1}C.$ Now let us consider an element $z \in D(CT_{1}C).$ Then for all $x \in D(T_{0}),$ we have $\langle T_{0}x, z \rangle = \langle CT_{1}^{*}Cx, z \rangle = \langle T_{1}Cz, Cx\rangle = \langle x, CT_{1}Cz\rangle.$ It shows that $z \in D(T_{0}^{*})$ which yields that $T_{0}^{*} = CT_{1}C.$ Therefore, $$R(T) = R(CT_{1}) = R(CT_{1}C) = R(T_{0}^{*})$$ is closed.
 \end{proof}
 \begin{remark}
As a consequence of Theorem \ref{thm 9}, one obtains
\[
(T_{1}T)^{*}=T^{*}T_{1}^{*},
\]
where \(T\) is a densely defined closed antilinear operator on \(H\) satisfying
\[
\dim R(T)^{\perp}<\infty,
\]
and \(T_{1}\) is a densely defined closed linear (or antilinear) operator on \(H\) such that \(T_{1}T\) is densely defined. The proof follows along the same lines as that of Corollary 1 in \cite{martin}.

Conversely, the condition
\[
\dim R(T)^{\perp}<\infty
\]
also follows from the identity
\[
(T_{1}T)^{*}=T^{*}T_{1}^{*}
\]
whenever this equality holds for all densely defined closed antilinear operators \(T_{1}\) and \(T\) on \(H\). The argument is analogous to the proof of Theorem 2 in \cite{jawvan}.
\end{remark}

 Antilinear symmetric operators were first introduced in \cite{IanK}, where author was referred to as conjugate-linear operators. In particular, Theorem~B of \cite{IanK} states that every maximal antilinear symmetric operator is necessarily an antilinear self-adjoint operator on a Hilbert space.

\begin{definition}
Let $T$ be a densely defined antilinear operator on $H$. Then $T$ is called \emph{antilinear symmetric} if $T \subset T^{\sharp}$. If $T = T^{\sharp}$, then $T$ is called \emph{antilinear self-adjoint}.
\end{definition}

By Theorem~13.12 and Lemma~13.13 of \cite{MR2953553}, every densely defined positive symmetric linear operator $T$ on $H$ admits a positive self-adjoint extension on $H$. Moreover, the minimal and maximal such extensions are known as the Krein--von Neumann extension and the Friedrichs extension of $T$, respectively.

 Theorem~B in \cite{IanK} shows that every densely defined  antilinear symmetric operator on a Hilbert space also admits an antilinear self-adjoint extension. It is therefore natural to ask  every densely defined  positive antilinear symmetric operator has a positive antilinear self-adjoint extension in the same Hilbert space and investigate the existence and characterization of the minimal and maximal positive antilinear  self-adjoint extensions of a densely defined antilinear positive symmetric operator.

 \begin{definition}\label{def 10}
 A densely defined closed antilinear operator $T$ on $H$ is called antilinear normal if $TT^{\sharp} = T^{\sharp}T.$
 \end{definition}
 The above Definition \ref{def 10} is inspired by the concept of bounded $C$-normal operators $S$ on $H$ characterized by
\[
(CS)^{\sharp}(CS) = (CS)(CS^{\sharp})
\]
(\cite{ptak2019c}). Analogously, we define a densely defined closed linear operator $S$ on $H$ to be $C$-normal (here $C$ is a conjugation on $H$) if
\[
(CS)(CS)^{\sharp} = (CS)^{\sharp}(CS).
\]
\begin{lemma}\label{lemma 11}
Let $S$ be a densely defined closed linear operator on $H$ and $C$ be a conjugation on $H$. $S$ is $C$-normal if and only if $$\|(CS)^{\sharp}x\| = \|CSx\|,$$ for all $x \in D(CS) = D(CS)^{\sharp}.$
\end{lemma}
\begin{proof}
First consider the condition $(CS)^{\sharp}(CS)= (CS)(CS)^{\sharp}.$ Then $D(S^{*}S) = D(CSS^{*}C).$ When $p \in D((CS)(CS)^{\sharp}),$ then
$$\|(CS)^{\sharp}p\|^{2} = \langle (CS)(CS)^{\sharp}p, p\rangle = \langle (CS)^{\sharp}(CS)p, p\rangle = \|(CS)p\|^{2}.$$ Again $D((CS)(CS)^{\sharp}) = D(S^{*}S)$ is a core of $S.$ Let $u \in D(CS) =D(S).$ There exists a sequence $\{u_{n}\}$ in $D((CS)(CS)^{\sharp})$ such that $u_{n} \to u$ and $Su_{n} \to Su,$ as $n \to \infty.$ It follows that 
\begin{align}\label{equ 7}
\|(CS)u\|^{2} = \|Su\|^{2} = \lim_{n \to \infty}\|Su_{n}\|^{2}= \lim_{n \to \infty}\|(CS)^{\sharp}u_{n}\|^{2} = \lim_{n \to \infty}\|S^{*}Cu_{n}\|^{2}.
\end{align}
Since $Cu_{n} \to Cu$ as $n \to \infty$ and $\{S^{*}Cu_{n}\}$ is Cauchy, so $S^{*}Cu_{n} \to S^{*}Cu,$ as $n \to \infty.$ By the identity (\ref{equ 7}), we get $$\|(CS)u\| = \|(CS)^{\sharp}u\|,$$ for all $u \in D(CS) \subset D(CS)^{\sharp}$ which also yields that $D(CS)^{\sharp}$ is dense. Let us consider an element $v \in D(CS)^{\sharp}.$ Then $Cv \in D(S^{*}).$ Moreover, $D(SS^{*})$ is also a core of $S^{*}.$ It is evident that the existence of a sequence $\{Cv_{n}\}$ in $D(SS^{*})$ such that $Cv_{n} \to Cv$ and $S^{*}Cv_{n} \to S^{*}Cv,$ as $n \to \infty.$ Thus,
\begin{align}\label{equ 8}
\|(CS)^{\sharp}v\| = \lim_{n \to \infty} \|(CS)^{\sharp}v_{n}\|= \lim_{n \to \infty} \|(CS)v_{n}\| = \lim_{n \to \infty} \|Sv_{n}\|.
\end{align}
Since $S$ is closed which implies that $\|(CS)^{\sharp}v\| = \|Sv\| = \|(CS)v\|.$ It also concludes that $D(CS)^{\sharp} \subset D(CS).$ Hence, $$\|(CS)^{\sharp}x\| = \|(CS)x\|,$$ for all $x \in D(CS)^{\sharp} = D(CS).$\\
To prove conversely, consider the identity $$\|(CS)q\| = \|(CS)^{\sharp}q\|,$$ for all $q \in D(CS) = D(CS)^{\sharp}.$ Let $z \in D(S^{*}S) \subset D(S)= D(CS).$ We obtain
 \begin{align}\label{equ 9}
\langle S^{*}Sz, z \rangle =  \langle CSz, CSz\rangle = \langle S^{*}Cz, S^{*}Cz\rangle.
 \end{align}
Now, we claim that $S^{*}Cz \in D(S).$ Taking a linear functional $$f(y) = \langle S^{*}y, S^{*}Cz \rangle,$$ for all $y \in D(SS^{*}).$ We get $w \in H$ with $Cw = y$ which shows that $w \in D(SS^{*}C) \subset D(CS.)$ Again,
\begin{align*}
f(y) &=  \langle S^{*}Cw, S^{*}Cz \rangle\\ 
&= \frac{1}{4} \bigl[ \|S^{*}C(w +z)\|^{2} - \|S^{*}C(w-z)\|^{2} +i\|S^{*}C(w-iz)\|^{2} -i \|S^{*}C(w + iz)\|^{2}\bigr]\\
&= \frac{1}{4} \bigl[ \|CS(w +z)\|^{2} - \|CS(w-z)\|^{2} +i\|CS(w-iz)\|^{2} -i \|CS(w + iz)\|^{2}\bigr]\\
&=\frac{1}{4} \bigl[ \|S(z + w)\|^{2} - \|S(z-w)\|^{2} +i\|S(z+ iw)\|^{2} -i \|S(z-iw)\|^{2}\bigr]\\
&=\langle Sz, Sw\rangle = \langle S^{*}Sz, Cy \rangle.
\end{align*}
So, $f$ is bounded which shows that $S^{*}Cz \in D(S).$ It means that $D(S^{*}S) \subset D(CSS^{*}C).$ By equation (\ref{equ 9}), we get $$\langle S^{*}Sz, z \rangle = \langle Cz, SS^{*}Cz\rangle = \langle CSS^{*}Cz, z \rangle.$$ We know that $D(S^{*}S)$ is dense in $H$ which leads to show that $S^{*}S \subset CSS^{*}C.$\\ Now, we will show that $D(CSS^{*}C) = D(SS^{*}C) \subset D(S^{*}S).$ Let $h \in D(SS^{*}C) \subset D(S^{*}C) = D(S).$ So, $$\|Sh\|^{2} = \|S^{*}Ch\|^{2}= \langle h, CSS^{*}Ch\rangle.$$ For all $k \in D(CSS^{*}C),$ consider a linear functional $g(k) = \langle Sk, Sh\rangle.$ Then
\begin{align*}
g(k)
&= \frac{1}{4} \Big[
\|S(k+h)\|^{2} - \|S(k-h)\|^{2}
+ i\,\|S(k+ih)\|^{2} - i\,\|S(k-ih)\|^{2}
\Big] \\
&= \frac{1}{4} \Big[
\langle k+h, \, C S S^{*} C (k+h) \rangle
- \langle k-h, \, C S S^{*} C (k-h) \rangle \\
&\qquad\quad
+ i\,\langle k+ih, \, C S S^{*} C (k+ih) \rangle
- i\,\langle k-ih, \, C S S^{*} C (k-ih) \rangle
\Big]\\
&= \langle SS^{*}Ch, Ck \rangle.
\end{align*}
It confirms that $g$ is bounded. Thus, $Sh \in D(S^{*})$ which implies that $D(CSS^{*}C) \subset D(S^{*}S).$ Therefore, $$(CS)(CS)^{\sharp}= CSS^{*}C = S^{*}S= (CS)^{\sharp}(CS).$$
\end{proof}
 \begin{theorem}\label{thm 12}
 Let $T$ be a densely defined closed antilinear operator on $H$.Then $T$ is antilinear normal if and only if $\|Tx\| = \|T^{\sharp}x\|,$ for all $x \in D(T) = D(T^{\sharp}).$
 \end{theorem}
 \begin{proof}
 We express $T = CT_{1},$ where $T_{1} = CT$ is a densely defined closed linear operator on $H$ and $C$ is a conjugation on $H.$  By Lemma \ref{lemma 11}, we find that $T = CT_{1}$ is antilinear normal if and only if $(CT_{1})(CT_{1})^{\sharp} = (CT_{1})^{\sharp}(CT_{1})$ if and only if $\|Tx\| = \|CT_{1}x\| = \|(CT_{1})^{\sharp}x\| = \|T^{\sharp}x\|,$ for all $x \in D(T) = D(T^{\sharp}).$
 \end{proof}
 \begin{example}
Let $\phi \in L^{2}(\mathbb{R})$. The multiplication operator $M_{\phi}$ on $L^{2}(\mathbb{R})$ is defined by
\[
(M_{\phi}f)(t)=\phi(t)f(t),
\]
for all
$f\in D(M_{\phi})
=
\{g\in L^{2}(\mathbb{R}) : \phi g\in L^{2}(\mathbb{R})\}.$
By Example 3.8 \cite{MR2953553}, the domain $D(M_{\phi})$ is dense in $L^{2}(\mathbb{R})$, and the adjoint of $M_{\phi}$ satisfies
$(M_{\phi})^{*}=M_{\overline{\phi}},$
where $\overline{\phi}$ denotes the complex conjugate of $\phi$.

Now consider the conjugation $C$ on $L^{2}(\mathbb{R})$ defined by
$C(h)=\overline{h}, \text{ for all } h\in L^{2}(\mathbb{R}).$
Then the antilinear operator $T=CM_{\phi}$ is antilinear self-adjoint. Indeed, for every
\[
f\in D(T^{\sharp})
=
D(M_{\overline{\phi}}C)
=
D(M_{\phi})
=
D(T),
\]
we have
$T^{\sharp}f
=
\overline{\phi f}
=
Tf.$
\end{example}
 \begin{example}
 Suppose that $a,b\in\mathbb{R}$ with $a<b$. Define the linear operator $S$ on $L^{2}(a,b)$ by
\[
Sf=-if',
\]
for all
$f\in D(S)=\{h\in AC[a,b]: h'\in L^{2}(a,b),\ h(a)=h(b)=0\},$
which is denoted by $H_{0}^{1}(a,b)$. Here, $AC[a,b]$ denotes the collection of all absolutely continuous complex-valued functions on $[a,b]$. Then the densely defined operator $S$ is closed and satisfies $S\subset S^{*}$, where the adjoint operator $S^{*}$ is given by
\[
S^{*}f=-if',
\]
with domain
$D(S^{*})=H^{1}(a,b)=\{h\in AC[a,b]: h'\in L^{2}(a,b)\},$
see Example 1.4 \cite{MR2953553}.

Now let $C$ be the conjugation on $L^{2}(\mathbb{R})$ defined by
$C(h)=\overline{h}, \text{ for all } h\in L^{2}(\mathbb{R}).$
Then the antilinear operator $T=CS$ is densely defined and closed. However, $T$ is not antilinear normal since
\[
D(T)=D(S)\subset D(S^{*})=D(S^{*}C)=D(T^{\sharp}).
\]

On the other hand, if the operator $S$ is defined on
\[
H^{1}(\mathbb{R})
=
\left\{
h\in L^{2}(\mathbb{R})
:
h\in AC[a,b]\ \text{for all}\ a,b\in\mathbb{R},
\ h'\in L^{2}(\mathbb{R})
\right\},
\]
then $S$ is self-adjoint; see Example 1.7 \cite{MR2953553}. Again, consider the conjugation $C$ on $L^{2}(\mathbb{R})$ given by
 $C(h)=\overline{h}, \text{ for all } h\in L^{2}(\mathbb{R}).$
In this case, the antilinear operator $T=CS$ is antilinear normal. Indeed, for every
\[
f\in D(T)=D(S)=D(S^{*})=D(S^{*}C)=D(T^{\sharp}),
\]
we have $\|Tf\|
=
\|Sf\|
=
\|f'\|
=
\|S^{*}Cf\|
=
\|T^{\sharp}f\|.$
 \end{example}
 \begin{example}
Define the linear operator
$S=-\frac{d^{2}}{dx^{2}}$
with domain
\[
D(S)=\{f\in H^{2}(\mathbb{R}) : f(0)=f'(0)=0\},
\]
where
$H^{2}(\mathbb{R})
=
\{f\in C^{1}(\mathbb{R}) : f'\in H^{1}(\mathbb{R})\},$
and $C^{1}(\mathbb{R})$ denotes the collection of all continuously differentiable complex-valued functions on $\mathbb{R}$. By Example 3.3 \cite{MR2953553}, the operator $S$ is symmetric and
\[
D(S^{*}) = H^{2}(-\infty,0)\oplus H^{2}(0,\infty).
\]

Now consider the conjugation $C$ on $L^{2}(\mathbb{R})$ defined by
$C(h)=\overline{h}, \text{ for all } h\in L^{2}(\mathbb{R}).$
Then the induced antilinear operator $T=CS$ is antilinear symmetric. Indeed, for every
\[
f\in D(T)=D(S)\subset D(S^{*})=D(S^{*}C)=D(T^{\sharp}),
\]
we have $Tf=-\overline{f''}=S^{*}Cf=T^{\sharp}f.$
\end{example}
 
The following Theorem \ref{thm 13} establishes the polar decomposition of a densely defined, closed, unbounded antilinear operator, and offers a characterization of antilinear normal operators on the Hilbert space $H$.
\begin{theorem}\label{thm 13}
Let $T$ be a densely defined closed antilinear operator from $H$ to $K$. Then
\[
T = U_T \, |T|,
\]
where $U_T$ is a partial antilinear isometry with initial space $\overline{R(|T|)}$ and final space $\overline{R(T)}$. 
\end{theorem}
\begin{proof}
We write $T = CT_{1},$ where $T_{1} = CT$ is also a densely defined closed operator from $H$ to $K$ and $C$ is a conjugation on $K.$ So, $T^{\sharp}T = T_{1}^{*}T_{1}$ is a positive self-adjoint operator on $H.$ We define $\vert T \vert$ as $\vert T_{1} \vert$ and the polar decomposition of a densely defined closed operator, we get $T_{1} = U_{T_{1}} \vert T \vert,$ where $U_{T_{1}}$ is a partial isometry with initial space $\overline{R(\vert T \vert)}$ and final space $\overline{R(T_{1})}.$\\
So, $T = CU_{T_{1}}\vert T \vert = U_{T} \vert T \vert,$ where $U_{T} = CU_{T_{1}}.$ Thus, $\|U_{T} x\| = \|x\|,$ for all $x \in \overline{R(\vert T \vert)}.$ It follows that $$N(U_{T}) = N(CU_{T_{1}}) = N(U_{T_{1}}) = N(\vert T \vert).$$ Let $u \in \overline{R(T)}.$ Then $Cu \in \overline{R(CT)} = \overline{R(T_{1})}.$ There exists an element $v \in \overline{R(\vert T \vert)}$ such that $Cu = U_{T_{1}}v$ which implies that $u = U_{T}v.$ Now, it is evident that $U_{T}$ is a partial antilinear isometry with initial space $\overline{R(\vert T \vert)}$ and final space $\overline{R(T)}.$
\end{proof}
\begin{corollary}\label{cor 14}
Let $T$  be an antilinear normal operator on $H.$ Then
\[
U_T \, |T| = |T| \, U_T.
\]
\end{corollary}
\begin{proof}
By the definition of antilinear normal operator $T,$ we have $TT^{\sharp} = T^{\sharp}T$ is a positive self-adjoint operator on $H.$ So, $\vert T \vert = (T^{\sharp}T)^{\frac{1}{2}} = (TT^{\sharp})^{\frac{1}{2}} = \vert T^{\sharp}\vert,$ where $\vert T^{\sharp}\vert = \vert CT^{\sharp}\vert$ and $C$ is a conjugation on $H.$ Now, we claim that $(U_{T})^{\sharp} = U_{T^{\sharp}}.$\\
By taking adjoint in the both sides of the identity $U_{T} = CU_{T_{1}},$ and $CT = T_{1} = U_{T_{1}} \vert T\vert,$ we get $(U_{T})^{\sharp} = U_{T_{1}^{*}}C.$ Again,
\begin{align}\label{equ 10}
(U_{T})^{\sharp}U_{T} = U_{T_{1}^{*}}U_{T_{1}} = P_{\overline{R(\vert T \vert)}} = P_{\overline{R(T^{\sharp})}}.
\end{align}
By the identity (\ref{equ 10}), we obtain that 
\begin{align}\label{equ 11}
(U_{T})^{\sharp} T= (U_{T})^{\sharp} U_{T}\vert T \vert = \vert T \vert,
\end{align}
and 
\begin{align}\label{equ 12}
(U_{T})^{\sharp}T(U_{T})^{\sharp} = \vert T \vert (U_{T})^{\sharp} = (U_{T} \vert T \vert)^{\sharp} = T^{\sharp}.
\end{align}
Hence, $TT^{\sharp} = U_{T} \vert T\vert \vert T \vert (U_{T})^{\sharp} = (U_{T}\vert T \vert (U_{T})^{\sharp} )^{2}.$   The positive operator $U_{T}\vert T \vert (U_{T})^{\sharp}$ is self-adjoint because $(U_{T}\vert T \vert (U_{T})^{\sharp})^{*} = (\vert T\vert (U_{T})^{\sharp})^{\sharp} (U_{T})^{\sharp} = T(U_{T})^{\sharp} = U_{T}\vert T \vert (U_{T})^{\sharp}.$ It shows that 
\begin{align}\label{equ 13}
\vert T^{\sharp}\vert = U_{T}\vert T \vert (U_{T})^{\sharp}.
\end{align}
For the relations (\ref{equ 13}) and (\ref{equ 10}), we conclude that 
\begin{align}\label{equ 14}
(U_{T})^{\sharp}\vert T^{\sharp}\vert = \vert T \vert (U_{T})^{\sharp}.
\end{align}
The identity (\ref{equ 14}) and the polar decomposition on $T^{\sharp}$ confirm that $$U_{T^{\sharp}}\vert T^{\sharp} \vert = T^{\sharp} = (U_{T}\vert T \vert)^{\sharp} = \vert T \vert (U_{T})^{\sharp} = (U_{T})^{\sharp} \vert T^{\sharp}\vert.$$
Then $U_{T^{\sharp}} = (U_{T})^{\sharp}$ on $\overline{R(\vert T^{\sharp}\vert)}.$ Let us consider an element $p \in (\overline{R(\vert T^{\sharp}\vert)})^{\perp} = N(T^{\sharp}).$ Then $U_{T^{\sharp}}p = 0,$ and $\|(U_{T})^{\sharp}p\|^{2} = \langle U_{T}(U_{T})^{\sharp}p, p\rangle = \langle CU_{T_{1}}U_{T_{1}^{*}}Cp, p\rangle = 0.$ It proves that $(U_{T})^{\sharp} = U_{T^{\sharp}}.$\\ Since $T$ is antilinear normal.  By the identities (\ref{equ 14}),(\ref{equ 10}) and $U_{T^{\sharp}} = (U_{T})^{\sharp}$, we establish that 
$$U_{T} \vert T \vert = U_{T} (U_{T})^{\sharp}\vert T^{\sharp} \vert U_{T} = \vert T^{\sharp}\vert U_{T} = \vert T \vert U_{T}.$$
\end{proof}
\begin{theorem}\label{thm 15}
Let $T$ be a densely defined closed antilinear operator on $H$  and $C$ be a conjugation on $H.$ Then $T$ is antilinear normal if and only if $C\vert T\vert C = \vert T^{\sharp}C\vert.$
\end{theorem}
\begin{proof}
We express $T = CT_{1},$ where $T_{1}= CT$ is a densely defined closed operator on $H.$ In order to verify the identity $C\vert T\vert C = \vert T^{\sharp}C\vert,$ it is enough to show that $$C\vert T_{1}\vert C = \vert T_{1}^{*} \vert.$$
Since, $T$ is antilinear normal, then $\vert T_{1}\vert^{2} = T_{1}^{*}T_{1} = CT_{1}T_{1}^{*}C = C\vert T_{1}^{*} \vert^{2} C.$ Again, $(C\vert T_{1} \vert C)^{2} = C\vert T_{1} \vert^{2}C.$ Now we claim that the positive operator $C \vert T_{1} \vert C$ is self-adjoint.\\
It is evident that $D(\vert T_{1}^{*}\vert^{2}) = D(C\vert T_{1}\vert^{2} C) \subset D(\vert T_{1} \vert^{2}C) \subset D(\vert T_{1} \vert C) = D(C \vert T_{1} \vert C).$ So, $(C\vert T_{1}\vert C)^{*}$ exits. Let $y \in D((C\vert T_{1}\vert C)^{*}).$ Then for all $x \in D(\vert T_{1}\vert C),$ we get
\begin{align}\label{equ 15}
\langle x, (C\vert T_{1}\vert C)^{*}y \rangle = \langle C\vert T_{1}\vert Cx, y \rangle =  \overline{\langle \vert T_{1} \vert Cx, Cy\rangle}.
\end{align}
Then the linear functional $f(z) = \langle \vert T_{1} \vert z, Cy\rangle,$ for all $z \in D(\vert T_{1} \vert)$ is bounded because $z = Cw,$ for some $w \in H.$ Thus, $Cy \in D(\vert T_{1} \vert).$ By the identity (\ref{equ 15}), we obtain $\langle x, (C\vert T_{1}\vert C)^{*}y \rangle = \langle x, C\vert T_{1} \vert Cy\rangle,$ for all $x \in D(\vert T_{1}\vert C).$ Hence, $(C\vert T_{1}\vert C)^{*} \subset (C\vert T_{1}\vert C).$\\
Let $u \in D(C \vert T_{1}\vert C).$ Then for all $v \in D(C\vert T_{1}\vert C),$ $g(v) = \langle C\vert T_{1}\vert Cv, u\rangle = \langle v, C\vert T_{1}\vert Cu \rangle$ is bounded linear functional. It follows that $u \in D((C\vert T_{1}\vert C)^{*}).$ It means $(C\vert T_{1}\vert C)$ is self-adjoint. Therefore, $\vert T^{\sharp}C\vert = \vert T_{1}^{*}\vert = C\vert T_{1}\vert C = C\vert T \vert C.$\\
To prove the converse, consider $C\vert T_{1}\vert C = \vert T_{1}^{*}\vert.$ Then, $$T^{\sharp}T= (CT_{1})^{\sharp}(CT_{1}) = \vert T_{1}\vert^{2} = CT_{1}T_{1}^{*}C=TT^{\sharp}.$$
\end{proof}
\begin{remark}\label{remark 16}
It is an immediate consequence of Theorem \ref{thm 12} that $T^{n}(T^{\sharp})^{n} = (T^{\sharp})^{n}T^{n},$ where $n \in \mathbb{N},$ for every antilinear normal operator $T$ on $H$. We also notice that two antilinear normal operators $T_{1}$ and $T_{2}$ in the Hilbert space $H$ with $T_{1} \subset T_{2}$ implies that $T_{1} = T_{2}.$ 
\end{remark}

\begin{definition}\label{def 17}
An antilinear operator $T$ is called an antilinear subnormal operator on $H$ if it admits an antilinear normal extension $N$ on $K$ in the sense that there exists a Hilbert space $K$ containing $H$ as a closed subspace in $K$ such that $Nh = Th,$ for all $h \in D(T).$\\
A antilinear normal extension $N$ on $K$ of an antilinear subnormal operator $T$ on $H$ is a minimal antilinear normal extension if for any antilinear normal extension $N_{1}$ on $K_{1}$ of $T$ on $H$ with $T \subset N_{1} \subset N$ and $K_{1}$ reduces under $N$ implies $K_{1} = K$ and $N_{1} = N.$
\end{definition}
Some significant results concerning unbounded subnormal linear operators on Hilbert spaces were established in \cite{stochel1989normal}. In particular, Theorem~3 of \cite{stochel1989normal} provides necessary and sufficient conditions for a densely defined linear operator $T$ on $H$, under the assumption that $T(D(T)) \subset D(T)$, to be subnormal. 

Furthermore, in \cite{stochel3}, the authors proved the existence of a minimal normal extension for a densely defined linear subnormal operator $T$ on $H$. However, their approach does not directly extend to the case of unbounded antilinear subnormal operators on complex Hilbert spaces. 

The following theorem establishes necessary and sufficient conditions for the existence of a minimal antilinear normal extension of an antilinear subnormal operator $T$ on $H$.
\begin{theorem}\label{thm 18}
Let $T$ be an antilinear subnormal operator on $H$, and let $N$ be an antilinear normal extension of $T$ acting on a Hilbert space $K \supset H$. Then $N$ is a minimal antilinear normal extension of $T$ if and only if the linear span
\[
G = \operatorname{span}\bigl\{ (N^{\sharp})^{j} N^{i} x : i,j \in \mathbb{N} \cup \{0\},\; x \in D^{\infty}(T) \bigr\}
\]
is dense in $K$, where
\[
D^{\infty}(T) = \bigcap_{n=1}^{\infty} D(T^{n}).
\]
\end{theorem}
\begin{proof}
First it is evident to the existence of $(N^{\sharp})^{j}N^{i}x,$ for all $x \in D^{\infty}(T)$ because $x \in D^{\infty}(T) \subset D(T^{j + i})$ implies $N^{i}x = T^{i}x \in D(T^{j}) \subset D(N^{j}) = D(N^{\sharp})^{j}.$\\
Suppose $G$ is not dense in $K.$ Then $\overline{G} + G^{\perp} = K,$ where $G^{\perp} \neq \{0\}.$\\
Again, It shows from the definition of $G$ that $N(G) \subset G,$ and $N^{\sharp}(G) \subset G.$ Let $u \in D(N) \cap G^{\perp}.$ Then for all $x \in G \subset D(N^{\sharp}) = D(N),$ we have $\langle Nu, x \rangle = \langle N^{\sharp}x, u \rangle = 0.$ So, $N(D(N) \cap G^{\perp}) \subset G^{\perp}.$ Similarly, we can show the identity
\begin{align}\label{equ 16}
N^{\sharp}(D(N) \cap G^{\perp}) \subset G^{\perp}.
\end{align}
Now we claim the relation $N(D(N) \cap \overline{G}) \subset \overline{G}.$ For fixed $z \in D(N) \cap \overline{G},$ and for all $y \in D(N) \cap G^{\perp},$  we get from the identity (\ref{equ 16}) that $$\langle Nz, y \rangle = \langle N^{\sharp}y, z \rangle = 0.$$ Thus, $N(D(N) \cap \overline{G}) \subset (D(N) \cap G^{\perp})^{\perp} = \overline{G},$ since $D(N)$ is dense in $K.$ This shows that $\overline{G}$ is a reduced subspace for $N.$ Hence, $N_{\vert{\overline{G}}}$ is also an antilinear normal extension of $T$ and $N_{\vert{\overline{G}}} \subset N.$ But $N$ is a minimal antilinear normal extension of $T$ which confirms that our assumption is wrong. Therefore, $G$ is dense in $K.$\\
To prove the converse, assume that $G$ is dense in $K.$ Suppose $N$ is not a minimal extension. Then there is an antilinear normal operator $N_{1}$ on $K_{1}$ such that $T \subset N_{1} \subset N$ and $D(T) \subset K_{1} \subset K$ with $K_{1}$ is a reduced space for $N.$\\
Again, $N_{1} \subset N_{\vert{K_{1}}}$ confirms that $D(N_{\vert{K_{1}}})^{\sharp} \subset D(N_{1})^{\sharp} =  D(N_{1}) \subset D(N_{\vert{K_{1}}}).$ Let $z \in D(N_{\vert{K_{1}}}).$ Then for all $x \in D(N_{\vert{K_{1}}}),$ the linear functional $$f(x) = \overline{\langle N_{\vert{K_{1}}}x, z \rangle} = \overline{\langle N^{\sharp}z, x \rangle}$$ is bounded and $(N_{\vert{K_{1}}})^{\sharp}z = N^{\sharp}z.$ So, $D(N_{\vert{K_{1}}})^{\sharp} = D(N_{\vert{K_{1}}}).$ This shows that $\| (N_{\vert{K_{1}}})^{\sharp}w\| = \|Nw\| = \|N_{\vert{K_{1}}}w\|,$ for all $w \in D(N_{\vert{K_{1}}})^{\sharp}.$ Hence, $N_{\vert{K_{1}}}$ is antilinear normal on $K_{1}.$ By Remark \ref{remark 16}, we get $N_{1} = N_{\vert{K_{1}}}.$ We also obtain that $$(N^{\sharp})^{j}N^{i}u= ((N_{\vert{K_{1}}})^{\sharp})^{j}(N_{\vert{K_{1}}})^{i}u,$$ for all $u \in D^{\infty}(T).$ Therefore, $G$ is in $K_{1}.$ Since, $G$ is dense in $K$ which establishes that $K_{1} = K$ and $N_{1} = N_{\vert{K_{1}}} = N.$ Therefore, $N$ is a minimal antilinear normal extension of $T.$
\end{proof}
The classical Toeplitz--Hausdorff theorem asserts that the numerical range of a linear operator $T$ on a Hilbert space $H$, defined by
\[
W(T) = \{ \langle Tx, x \rangle : x \in D(T),\ \|x\| = 1 \},
\]
is a convex subset of $\mathbb{C}$. 

Recently, in \cite{Vmuller}, the authors established the convexity of the numerical range for bounded antilinear operators on complex Hilbert spaces, provided that the dimension of the domain space is greater than one. It is worth noting that convexity fails when one considers a conjugation operator on a one-dimensional Hilbert space.
The following theorem establishes the convexity of the numerical range of an antilinear operator $T$ on $H$ whenever $\dim H > 1$.
\begin{theorem}\label{thm 19}
    Let $T$ be an antilinear operator on $H$ with $\dim{H} > 1.$ Then $W(T)$ is convex.
\end{theorem}
\begin{proof}
First consider two elements $\alpha_{1}, \alpha_{2} \in W(T).$ Then $\alpha_{1} = \langle Tx_{1}, x_{1}\rangle$ and $\alpha_{2} = \langle Tx_{2}, x_{2} \rangle,$ where $x_{1} , x_{2} \in D(T)$ with $\|x_{1}\| = \|x_{2}\| = 1.$ Choose the real numbers as $t,r.$ Define two functions as follows:
\begin{align*}
S_{1}(t, r) &= \langle T(tx_{1} + rx_{2}), tx_{1} + rx_{2}\rangle\\ & = t^{2} \alpha_{1} + tr \langle Tx_{1},x_{2}\rangle + tr\langle Tx_{2}, x_{1}\rangle + r^{2} \alpha_{2}.
\end{align*}
and 
\begin{align*}
S_{2}(t, r) &= \|tx_{1} +  r x_{2}\|^{2}\\
&=\langle tx_{1}+  rx_{2}, tx_{1} +  r x_{2}\rangle\\
&= t^{2}+ 2 t r Re\langle x_{1}, x_{2} \rangle + r^{2}.
\end{align*}
From $S_{1}$ and $S_{2},$ we get another function 
\begin{align*}
S(t, r) &= \frac{S_{1}(t, r)- \alpha_{2}S_{2}(t, r)}{\alpha_{1}- \alpha_{2}}\\
&= t^{2} + {\beta rt},
\end{align*}
where $\beta = \frac{\langle Tx_{1}, x_{2}\rangle + \langle Tx_{2}, x_{1}\rangle-2\alpha_{2}Re\langle x_{1}, x_{2}\rangle}{\alpha_{1} - \alpha_{2}}.$ When $S_{2}(t, r) = 1,$ then 
\begin{align*}
r = {-t Re\langle x_{1}, x_{2}\rangle \pm \sqrt{t^{2}(Re\langle x_{1}, x_{2}\rangle)^{2}- t^{2}+ 1}}.
\end{align*}
We observe that $r$ is real when $\vert t \vert \leq 1.$ Thus, we get the real solutions of $S_{2}(t, r) =1,$ when $\vert t \vert \leq 1.$  Again we get a new function in the domain $\vert t \vert \leq 1$ as
$$S_{3}(t) = S(t, r) = t^{2} + \beta ({-t Re\langle x_{1}, x_{2}\rangle + \sqrt{t^{2}(Re\langle x_{1}, x_{2}\rangle)^{2}- t^{2}+ 1}})t$$ is continuous. It follows that $S_{3}(0) = 0$ and $S_{3}(1) = 1.$ For $\lambda \in [0,1],$ there exists $t^{'} \in [0,1]$
 such that 
 \begin{align*}\frac{S_{1}(t^{'}, {-t^{'} Re\langle x_{1}, x_{2}\rangle + \sqrt{{t^{'}}^{2}(Re\langle x_{1}, x_{2}\rangle)^{2}- {t^{'}}^{2}+ 1}}) - \alpha_{2}}{\alpha_{1} - \alpha_{2}} = S_{3}(t^{'}) = \lambda.
 \end{align*}
 Hence, we conclude that there are two real numbers $t^{'} \in [0,1]$ and $r^{'} = {-t^{'} Re\langle x_{1}, x_{2}\rangle + \sqrt{{t^{'}}^{2}(Re\langle x_{1}, x_{2}\rangle)^{2}- {t^{'}}^{2}+ 1}}$ with $\|t^{'}x_{1} + r^{'}x_{2}\| = 1$ such that $$\langle T(t^{'}x_{1} + r^{'}x_{2}), t^{'}x_{1} + r^{'}x_{2}\rangle =  S_{1}(t^{'}, r^{'}) = \lambda \alpha_{1} + (1- \lambda)\alpha_{2}.$$ Therefore, $W(T)$ is convex.
\end{proof}
\begin{remark}
For any antilinear operator $T$ on $H$ with $\dim{H} > 1,$ $0 \in W(T).$ The proof proceeds analogously to that of Proposition 2.1 \cite{Vmuller}.  
\end{remark}
The Moore-Penrose inverse is named after E.H Moore and Roger Penrose, who independently introduced the concept in 1920 and 1955, respectively. The Moore-Penrose inverses of closed operators in Hilbert spaces were studied in paper \cite{MR007, MR0396607, MRARUP}.\\ In the following Theorem \ref{thm 22}, we characterize the basic properties of the Moore-Penrose inverses of closed antilinear operators in Hilbert spaces.
\begin{definition}
		Let $T$ be a closed antilinear operator from  $D(T) \subset H$ to  $K$. The Moore-Penrose inverse of $T$ is the map $T^{\dagger}: R(T) \oplus^{\perp} R(T)^{\perp} \to H$ defined by
		\begin{equation}\label{equ 17}
			T^{\dagger} y = 
			\begin{cases}
				({T}\vert_{N(T)^{\perp} \cap D(T)})^{-1}y  & \text{if} ~  y\in R(T)\\
				0    & \text{if}  ~ y\in R(T)^{\perp}.
			\end{cases}
		\end{equation}
    \end{definition}
    It can be shown that $T^{\dagger}$ is also a closed antilinear operator from $K$ to $H.$
\begin{theorem}\label{thm 22}
Let $T$ be a densely defined closed antilinear operator from $H$ to $K.$ Then the following statements hold:
\begin{enumerate}
\item $(T^{\sharp})^{\dagger} = (T^{\dagger})^{\sharp}.$
\item $(T^{\dagger})^{\dagger} = T.$
\item $(T^{\sharp}T)^{\dagger} = T^{\dagger}(T^{\sharp})^{\dagger}.$
\item $(TT^{\sharp})^{\dagger} = (T^{\sharp})^{\dagger}T^{\dagger}.$
\item $\vert T \vert^{\dagger} = \vert (T^{\sharp})^{\dagger}\vert ~ \text{ and }~ \vert T^{\dagger}\vert = \vert T^{\sharp}\vert^{\dagger}.$
\end{enumerate}

  Moreover, if $R(T)$ is closed, then:
  \begin{enumerate}
  \setcounter{enumi}{5}
  \item $T^{\dagger} = \vert T \vert^{\dagger}U_{T^{\sharp}}~ \text{and}~ (T^{\sharp})^{\dagger} = U_{T} \vert (T^{\sharp})^{\dagger} \vert$ is called the polar decomposition of $(T^{\sharp})^{\dagger}.$
  \item $R(T) = R(T^{\sharp})~ \text{ if and only if }~ TT^{\dagger} \supset T^{\dagger}T.$
  \end{enumerate}
\end{theorem}
\begin{proof}
We express $T = CT_{1},$ where $C$ is a conjugation on $K$ and $T_{1} = CT$ is a densely defined closed operator from $H$ to $K.$\\

$\it(1)$ It is evident that
\begin{align*}
D(T_{1}^{\dagger}C) &= N(T_{1}^{\dagger}C) \oplus N(T_{1}^{\dagger}C)^{\perp}\cap D(T_{1}^{\dagger}C)\\ & = R(CT_{1})^{\perp} \oplus \overline{R(CT_{1})} \cap D(T_{1}^{\dagger}C)\\ &=  R(CT_{1})^{\perp} \oplus  R(CT_{1})\\ &= D(CT_{1})^{\dagger} = D(T^{\dagger}).
\end{align*}
Let $p \in R(CT_{1})^{\perp} = N(T^{\sharp}) = N(T_{1}^{*}C).$ Then $Cp \in R(T_{1})^{\perp}.$ So, $(CT_{1})^{\dagger}p = 0 = T_{1}^{\dagger}Cp.$ Again, consider an element $q \in R(CT_{1}).$ We get $$w \in N(CT_{1})^{\perp} \cap D(CT_{1}) = N(T_{1})^{\perp} \cap D(T_{1})$$ such that $(CT_{1})^{\dagger}q = w.$ It follows that $T_{1}^{\dagger}Cq = w.$ Thus, $(CT_{1})^{\dagger} = T_{1}^{\dagger}C.$ We also obtain that $(T_{1}^{\dagger}C)^{\sharp} = C(T_{1}^{\dagger})^{*}.$ By Theorem 1.5(7) \cite{MR0396607}, we have $(T_{1}^{\dagger}C)^{\sharp} = C(T_{1}^{*})^{\dagger}.$  Observe that 
\begin{align*}
D(T_{1}^{*}C)^{\dagger} &= R(T_{1}^{*}C) \oplus R(T_{1}^{*}C)^{\perp}\\ &= R(T_{1}^{*}) \oplus R(T_{1}^{*})^{\perp}\\ & = D(T_{1}^{*})^{\dagger}= D(C (T_{1}^{*})^{\dagger}).
\end{align*}
Let $z \in R(T_{1}^{*}C)^{\perp} = R(T_{1}^{*})^{\perp}.$ So, $(T_{1}^{*}C)^{\dagger}z = 0 = C(T_{1}^{*})^{\dagger}z.$ For an element $v \in R(T_{1}^{*}C),$ we get $u \in N(T_{1}^{*}C)^{\perp} \cap D(T_{1}^{*}C)$ such that $T_{1}^{*}Cu = v.$ This implies that $Cu \in N(T_{1}^{*})^{\perp} \cap D(T_{1}^{*}).$ Thus, $(T_{1}^{*}C)^{\dagger}v = u = C(T_{1}^{*})^{\dagger}v$ which confirms that  $(T_{1}^{*}C)^{\dagger} = C(T_{1}^{*})^{\dagger}.$\\
Therefore,
$$(T^{\dagger})^{\sharp} = (T_{1}^{\dagger}C)^{\sharp} = C(T_{1}^{*})^{\dagger}= (T_{1}^{*}C)^{\dagger} = (T^{\sharp})^{\dagger}.$$
$\it(2)$ First we claim that $(T_{1}^{\dagger}C)^{\dagger} = CT_{1}.$ Observe that 
\begin{align*}
D(T_{1}^{\dagger}C)^{\dagger} &= R(T_{1}^{\dagger}C) \oplus R(T_{1}^{\dagger}C)^{\perp}\\ &= R(T_{1}^{\dagger}) \oplus R(T_{1}^{\dagger})^{\perp}\\ &= N(T_{1})^{\perp} \cap D(T_{1}) \oplus N(T_{1})\\ &= D(T_{1}).
\end{align*}
Consider an element $y \in R(T_{1}^{\dagger}C)^{\perp} = R(T_{1}^{\dagger})^{\perp} = N(T_{1}).$ So, $(T_{1}^{\dagger}C)^{\dagger}y = 0 = CT_{1}y.$
Again, taking an element $v \in R(T_{1}^{\dagger}C) = N(T_{1})^{\perp} \cap D(T_{1}),$ we get $$u \in N(T_{1}^{\dagger}C)^{\perp} \cap D(T_{1}^{\dagger}C) ~ \text{ and }~ Cu \in N(T_{1}^{\dagger})^{\perp} \cap D(T_{1}^{\dagger}) = R(T_{1})$$ such that $T_{1}^{\dagger}Cu = v.$ It shows that $CT_{1}v = u = (T_{1}^{\dagger}C)^{\dagger}v.$ Hence,
$$(T^{\dagger})^{\dagger} = (T_{1}^{\dagger}C)^{\dagger} = CT_{1} = T.$$
$\it(3)$ By Using Theorem 1.5(9) \cite{MR0396607}, we obtain that 
\begin{align*}
(T^{\sharp}T)^{\dagger} = (T_{1}^{*}T_{1})^{\dagger} = T_{1}^{\dagger}(T_{1}^{\dagger})^{*} = T^{\dagger}C(T_{1}^{*})^{\dagger} = T^{\dagger}(T^{\sharp})^{\dagger}.
\end{align*}
$\it(4)$ We observe that an element $y \in D(CT_{1}T_{1}^{*}C)^{\dagger} = R(CT_{1}T_{1}^{*}C) \oplus R(CT_{1}T_{1}^{*}C)^{\perp}$ if and only if $Cy \in R(T_{1}T_{1}^{*}C) \oplus R(T_{1}T_{1}^{*}C)^{\perp} = D(T_{1}T_{1}^{*}C)^{\dagger}.$ So, it follows that $$(CT_{1}T_{1}^{*}C)^{\dagger} = (T_{1}T_{1}^{*}C)^{\dagger}C.$$ Again,
$$D(T_{1}T_{1}^{*}C)^{\dagger} = D(T_{1}T_{1}^{*})^{\dagger}= D(C(T_{1}T_{1}^{*})^{\dagger}).$$ From the proof of part $(2)$ and Theorem 1.5(10) \cite{MR0396607}, we get $$(T_{1}T_{1}^{*}C)^{\dagger} = C(T_{1}T_{1}^{*})^{\dagger} = C(T_{1}^{*})^{\dagger}T_{1}^{\dagger}.$$ Therefore,
$$(TT^{\sharp})^{\dagger} = C(T_{1}^{*})^{\dagger}T_{1}^{\dagger}C = (T^{\sharp})^{\dagger}T^{\dagger}.$$

$\it(5)$ By definition, we have $\vert T \vert = \vert CT_{1}\vert =\vert T_{1}\vert.$ From Corollary 3.8 \cite{MR007}, we obtain that $$\vert T \vert^{\dagger} = \vert T_{1}\vert^{\dagger} = \vert (T_{1}^{*})^{\dagger} \vert = \vert (T_{1}^{\dagger})^{*}\vert= \vert C(T_{1}^{\dagger})^{*} \vert = \vert (T^{\sharp})^{\dagger}\vert.$$ The second identity is also valid because $$\vert T^{\dagger}\vert = \vert ((T^{\dagger})^{\sharp})^{\sharp}\vert = \vert T^{\sharp}\vert^{\dagger}.$$
$\it(6)$ By Corollary \ref{cor 8}, we can say that $R(T)$ is closed if and only if $R(T^{\sharp}) = R(T_{1}^{*})$ is closed if and only if $R(T_{1})$ is closed. From Lemma 2.6 \cite{MR0396607} and the statement $(U_{T})^{\sharp} = U_{T^{\sharp}}$ from Corollary \ref{cor 14}, we obatin that 
$$T^{\dagger} = T_{1}^{\dagger}C= \vert T_{1}\vert^{\dagger}U_{T_{1}^{*}}C = \vert T\vert^{\dagger} (U_{T})^{\sharp} = \vert T \vert^{\dagger}U_{T^{\sharp}}.$$ It shows that the polar decomposition of $$(T^{\sharp})^{\dagger} = (\vert T \vert^{\dagger}U_{T^{\sharp}})^{\sharp} =  (\vert T_{1} \vert^{\dagger}U_{T^{\sharp}})^{\sharp} = U_{T}\vert T\vert^{\dagger} = U_{T} \vert (T^{\sharp})^{\dagger}\vert.$$

$\it(7)$ Since, $TT^{\dagger} = T^{\dagger}T$ on $D(T).$ $N(T)^{\perp} \cap D(T) \subset R(T).$ Again $R(T)$ is closed which implies $R(T^{\sharp}) \subset R(T).$ Consider an element $x \in N(T).$ It shows that $TT^{\dagger}x = 0.$ Thus, $N(T) \subset R(T)^{\perp}.$ Hence, $R(T) = R(T^{\sharp}).$\\
To prove the converse, assume that $R(T) = R(T^{\sharp}).$ Then 
$$T^{\dagger}T \subset P_{N(T)^{\perp}} = P _{R(T)} = TT^{\dagger}.$$
\end{proof}
\section{Antilinear block operator matrices}
Let $\mathcal{A} = \begin{bmatrix}
  A & B\\
  F & E
  \end{bmatrix}$.  Here $A$ and  $E$ are antilinear operators on $H$ and $K$ respectively whereas $B$ and $F$ are  antilinear operators from $K$ to $H$ and $H$ to $K$ respectively.  Observe that
\[
D(\mathcal A)=\bigl(D(A)\cap D(F)\bigr)\oplus \bigl(D(B)\cap D(E)\bigr).
\]
Here, the operators \(A,B,F,\) and \(E\) associated with \(\mathcal A\) are defined as described earlier. Determining necessary and sufficient conditions for the closability---and, in particular, the closedness---of \(\mathcal A\) is of significant interest, as these properties play a fundamental role in establishing important analytical results such as the Closed Range Theorem, polar decomposition, and spectral properties of \(\mathcal A\).

Linear block operator matrices on Hilbert spaces have attracted considerable attention in recent years; see, for instance, \cite{MR2463978, majumdar2024hyers}. In the subsequent theorems, we explore several structural and analytical aspects of \(\mathcal A\) with the aim of characterizing its closability properties.\\
It is worth noting that the resolvent set of a closed antilinear operator \(T\) on a Hilbert space \(H\), denoted by \(\rho(T)\), consists of all complex numbers \(\lambda \in \mathbb{C}\) for which the real linear (additive) operator $T-\lambda$ is bijective, that is, both one-to-one and onto. Whenever the operator \(T-\lambda\) is bijective, its inverse $(T-\lambda)^{-1}$ is necessarily bounded on \(H\) by virtue of the Open Mapping Theorem for real Banach (equivalently, real Hilbert) spaces. The spectrum of \(T\), denoted by \(\sigma(T)\), is defined as the complement of the resolvent set \(\rho(T)\) in the complex plane \(\mathbb{C}\); that is,
$\sigma(T)=\mathbb{C}\setminus \rho(T).$
\begin{theorem}\label{thm 4.1}
Assume that
\[
D(A)\subset D(F), \quad \rho(A)\neq \emptyset,
\]
and that for some \(\mu\in \rho(A)\), the real linear operator \((A-\mu)^{-1}B\) is bounded on \(D(B)\), while the real linear operator \(F(A-\mu)^{-1}\) is closed on \(H\). Then the block operator matrix \(\mathcal A\) is closable if and only if its Schur complement
\[
S_{2}(\mu)=E-\mu-F(A-\mu)^{-1}B
\]
is closable. Moreover, in this case, the closure of \(\mathcal A\) is given by
\begin{equation}\label{equ 18}
\overline{\mathcal A}
=\mu+
\begin{bmatrix}
I & 0\\
F(A-\mu)^{-1} & I
\end{bmatrix}
\begin{bmatrix}
A-\mu & 0\\
0 & \overline{S_{2}(\mu)}
\end{bmatrix}
\begin{bmatrix}
I & \overline{(A-\mu)^{-1}{B}}\\
0 & I
\end{bmatrix}.
\end{equation}
\end{theorem}
\begin{proof}
Since the real linear operator $(A- \mu)^{-1}B$ is bounded on $D(B),$ then we can extend to $\overline{D(B)}$ to get a new real linear closed operator $\widehat{((A-\mu)^{-1}B)}$ on domain $\overline{D(B)}.$ So, $(A-\mu)^{-1}B$ is closable and the real linear closed operator $\overline{(A-\mu)^{-1}B}$ is bounded on domain $\overline{D(B)}.$ Then 
\[
\begin{bmatrix}
I & \overline{(A-\mu)^{-1}{B}}\\
0 & I
\end{bmatrix},
\qquad \text{and} \qquad
\begin{bmatrix}
I & 0\\
F(A-\mu)^{-1} & I
\end{bmatrix}.
\]
both are bounded invertible real linear operators on $H \times \overline{D(B)}$ and $H \times K,$ respectively. Since, $\mathcal{A}$ is closable. The closable linear real opeartor $\mathcal{A}- \mu$ can be expressed as
$$\mathcal{A-\mu} =  \begin{bmatrix}
I & 0\\
F(A-\mu)^{-1} & I
\end{bmatrix}
\begin{bmatrix}
A-\mu & 0\\
0 & {S_{2}(\mu)}
\end{bmatrix}
\begin{bmatrix}
I & (A-\mu)^{-1}{B}\\
0 & I
\end{bmatrix}.$$ So, $$\begin{bmatrix}
I & 0\\
-F(A-\mu)^{-1} & I
\end{bmatrix} \overline{(\mathcal{A}- \mu)} = \overline{\begin{bmatrix}
A-\mu & 0\\
0 & {S_{2}(\mu)}
\end{bmatrix}
\begin{bmatrix}
I & (A-\mu)^{-1}{B}\\
0 & I
\end{bmatrix}}.$$ 
Let $(0,y) \in \overline{G(S_{2}(\mu))}.$ Then there exists a sequence $\{x_{n}\} \in D(S_{2}(\mu))$ such that $x_{n} \to 0,$ and $S_{2}(\mu)x_{n} \to y,$ as $n \to \infty.$  Hence, $$\begin{bmatrix}
A-\mu & 0\\
0 & {S_{2}(\mu)}
\end{bmatrix}
\begin{bmatrix}
I & (A-\mu)^{-1}{B}\\
0 & I
\end{bmatrix} \begin{pmatrix}
-(A -\mu)^{-1}Bx_{n}\\
x_n
\end{pmatrix} \to \begin{pmatrix}
0\\
y
\end{pmatrix},$$ as $n \to \infty.$ Thus, $y = 0$ which confirms that $S_{2}(\mu)$ is closable.\\
Conversely, consider $S_{2}(\mu)$ is closable. Then
$$\mathcal{A} - \mu \subset \begin{bmatrix}
I & 0\\
F(A-\mu)^{-1} & I
\end{bmatrix}
\begin{bmatrix}
A-\mu & 0\\
0 & \overline{S_{2}(\mu)}
\end{bmatrix}
\begin{bmatrix}
I & (\overline{A-\mu)^{-1}{B}}\\
0 & I
\end{bmatrix}$$
is closable, so $\mathcal{A}$ is also closable. Furthermore,
 \begin{align*}\begin{bmatrix}
I & 0\\
-F(A-\mu)^{-1} & I
\end{bmatrix} \overline{(\mathcal{A}- \mu)} &= \overline{\begin{bmatrix}
A-\mu & 0\\
0 & {S_{2}(\mu)}
\end{bmatrix}
\begin{bmatrix}
I & (A-\mu)^{-1}{B}\\
0 & I
\end{bmatrix}}\\
&= \begin{bmatrix}
A-\mu & 0\\
0 & \overline{S_{2}(\mu)}
\end{bmatrix}
\begin{bmatrix}
I & \overline{(A-\mu)^{-1}{B}}\\
0 & I
\end{bmatrix}.
\end{align*}
Therefore,
$$\overline{\mathcal A}- \mu = \overline{\mathcal{A}- \mu}=
\begin{bmatrix}
I & 0\\
F(A-\mu)^{-1} & I
\end{bmatrix}
\begin{bmatrix}
A-\mu & 0\\
0 & \overline{S_{2}(\mu)}
\end{bmatrix}
\begin{bmatrix}
I & \overline{(A-\mu)^{-1}{B}}\\
0 & I
\end{bmatrix}.$$
\end{proof}
\begin{remark}\label{remark 4.2}
The conclusion of Theorem~\ref{thm 4.1} remains valid for every \(\lambda \in \rho(A)\). If for some fixed \(\mu \in \rho(A)\) the operator \((A-\mu)^{-1}B\) is bounded on \(D(B)\) and $F(A-\mu)^{-1}$ is closed.

Indeed, for all \(\lambda\) satisfying
\[
|\lambda-\mu|<\frac{1}{\|(A-\mu)^{-1}\|},
\]
the real linear operator
\[
F(A-\lambda)^{-1}
=
F\sum_{n=0}^{\infty}(\lambda-\mu)^n\bigl((A-\mu)^{-1}\bigr)^{n+1}
\]
is bounded. This follows from the assumption \(D(F)\subset D(A)\) together with the boundedness of the real linear operator \(F(A-\mu)^{-1}\). Consequently, \((A-\lambda)^{-1}B\) is also bounded on \(D(B)\).
\end{remark}
\begin{theorem}\label{thm 4.3}
Assume that \[
\|Fh\| \leq a \|Ah\|, ~\text{for all}~ h \in D(A)\subset D(F) ~\text{and}~ a\geq 0, \quad 0 \in \rho(A),
\]
and the linear operator \(A^{-1}B\) is bounded on \(D(B)\), while the Schur complement $S_{2}(0) = E -FA^{-1}B$ is closable. Then $R(\overline{\mathcal{A}})$ is closed if and only if $R(\overline{S_{2}(0)})$ is closed.
\end{theorem}
\begin{proof}
It is evident that $FA^{-1}$ is bounded in the domain $H.$
   \[
\overline{\mathcal A}
=
\begin{bmatrix}
I & 0\\
FA^{-1} & I
\end{bmatrix}
\begin{bmatrix}
A & 0\\
0 & \overline{S_{2}(0)}
\end{bmatrix}
\begin{bmatrix}
I & \overline{A^{-1}B}\\
0 & I
\end{bmatrix},
\]
where
\[
S_{2}(0)=E-FA^{-1}B.
\]

Further,
\[
\overline{\mathcal A}
\begin{pmatrix}
u\\
v
\end{pmatrix}
=
\begin{pmatrix}
Au+A\overline{A^{-1}B}v\\
Fu+F\overline{A^{-1}B}v+\overline{S_{2}(0)}v
\end{pmatrix},
\qquad
\begin{pmatrix}
u\\
v
\end{pmatrix}
\in D(\overline{\mathcal A}).
\]

We first show that
\[
\overline{R(\mathcal A)}
\subseteq
R(\overline{\mathcal A}),
\]
whenever \(R(\overline{S_{2}(0)})\) is closed. The reverse inclusion follows directly from the definition of a closable antilinear operator.

Let $\begin{pmatrix}
x\\
y
\end{pmatrix}
\in
\overline{R(\mathcal A)}.$
Then there exists a sequence
$$
\left\{
\begin{pmatrix}
x_{n}\\
y_{n}
\end{pmatrix}
\right\}
\subset
D(\mathcal A)
=
D(A)\oplus \bigl(D(B)\cap D(E)\bigr)$$
such that
\[
Ax_{n}+By_{n}\to x,
\qquad
Fx_{n}+Ey_{n}\to y,
\]
as \(n\to\infty\). Consequently,
$Fx_{n}+FA^{-1}By_{n}\to FA^{-1}x,$
and therefore,
$$(E-FA^{-1}B)y_{n}
=
S_{2}(0)y_{n}
\to
y-FA^{-1}x,$$
as \(n\to\infty\).

Since
$R(\overline{S_{2}(0)})
=
\overline{R(S_{2}(0))},$
there exists \(w\in D(\overline{S_{2}(0)})\) such that $\overline{S_{2}(0)}w
=
y-FA^{-1}x.$

Hence,
\[
\overline{\mathcal A}
\begin{pmatrix}
A^{-1}x-\overline{A^{-1}B}w\\
w
\end{pmatrix}
=
\begin{pmatrix}
x\\
y
\end{pmatrix},
\]
which implies that $
\begin{pmatrix}
x\\
y
\end{pmatrix}
\in
R(\overline{\mathcal A}).$
Thus, $\overline{R(\mathcal A)}
=
R(\overline{\mathcal A}),$
and consequently \(R(\overline{\mathcal A})\) is closed.

Conversely, assume that \(R(\overline{\mathcal A})\) is closed. We prove that
$\overline{R(S_{2}(0))}
\subseteq
R(\overline{S_{2}(0)}).$

Let
$y\in \overline{R(S_{2}(0))}.$
Then there exists a sequence \(\{x_{n}\}\subset D(S_{2}(0))\) such that
\[
(E-FA^{-1}B)x_{n}
=
S_{2}(0)x_{n}
\to y,
\]
as \(n\to\infty\). Therefore,
\[
\overline{\mathcal A}
\begin{pmatrix}
-\overline{A^{-1}B}x_{n}\\
x_{n}
\end{pmatrix}
=
\begin{pmatrix}
0\\
S_{2}(0)x_{n}
\end{pmatrix}
\to
\begin{pmatrix}
0\\
y
\end{pmatrix},
\]
as \(n\to\infty\).

Since \(R(\overline{\mathcal A})\) is closed, there exists
$\begin{pmatrix}
u_{0}\\
v_{0}
\end{pmatrix}
\in D(\overline{\mathcal A})$
such that
$Au_{0}+A\overline{A^{-1}B}v_{0}=0$
and

$Fu_{0}
+
F\overline{A^{-1}B}v_{0}
+
\overline{S_{2}(0)}v_{0}
=
y.$
Using the inequality
$\|F x\|\le a\|A x\|,$
we obtain $$Fu_{0}+F\overline{A^{-1}B}v_{0}=0.$$
Hence, $\overline{S_{2}(0)}{v_{0}}=y,$ which shows that
$y\in R(\overline{S_{2}(0)}).$ Therefore, $\overline{R(S_{2}(0))}
= R(\overline{S_{2}(0)}),$ thus $R(\overline{S_{2}(0)})$ is closed.
\end{proof}
\begin{theorem}\label{thm 4.4}
Assume that
\[
D(E)\subset D(B), \quad \rho(E)\neq \emptyset,
\]
and that for some \(\mu\in \rho(E)\), the real linear operator \((E-\mu)^{-1}F\) is bounded on \(D(F)\), while the real linear operator \(B(E-\mu)^{-1}\) is closed on \(K\). Then the block operator matrix \(\mathcal A\) is closable if and only if its Schur complement
\[
S_{1}(\mu)=A-\mu-B(E-\mu)^{-1}F
\]
is closable. Moreover, in this case, the closure of \(\mathcal A\) is given by
\begin{equation}\label{equ 19}
\overline{\mathcal A}
=\mu+
\begin{bmatrix}
I & B(E-\mu)^{-1}\\
0 & I
\end{bmatrix}
\begin{bmatrix}
\overline{S_{1}(\mu)} & 0\\
0 & E-\mu
\end{bmatrix}
\begin{bmatrix}
I & 0\\
\overline{(E-\mu)^{-1}{F}} & I
\end{bmatrix}.
\end{equation}
\end{theorem}
\begin{proof}
The proof proceeds analogously to that of Theorem~\ref{thm 4.1}.
\end{proof}
\begin{theorem}\label{thm 4.5}
Assume that \[
\|Bh\| \leq b \|Eh\|, ~\text{for all}~ h \in D(E)\subset D(B) ~\text{and}~ b\geq 0, \quad 0 \in \rho(E),
\]
and the linear operator \(E^{-1}F\) is bounded on \(D(F)\), while the Schur complement $S_{1}(0) = A -BE^{-1}F$ is closable. Then $R(\overline{\mathcal{A}})$ is closed if and only if $R(\overline{S_{1}(0)})$ is closed.
\end{theorem}
\begin{proof}
The proof follows along the same lines as the proof of Theorem \ref{thm 4.3}.
\end{proof}
\begin{theorem}\label{thm 4.6}
Assume that $$D(F) \subset D(A),~~ F \text{ is bijective },$$ 
$F^{-1}E$ is bounded on $D(E)$ and $(A-\mu)F^{-1}$ is closed. Then $\mathcal A$ is closable if and only if its quadratic complement $$T_{2}(\mu) =  B - (A-\mu)F^{-1}(E-\mu) $$ is closable for some (and hence for all) $\mu \in \mathbb {C}$. Moreover, in this case,  the closure $\overline{\mathcal A}$ is given by
\begin{equation}\label{equ 20}
    \overline{\mathcal A} = \mu + \begin{bmatrix}
        I & (A-\mu)F^{-1}\\
        0 & I
    \end{bmatrix}
    \begin{bmatrix}
    0 & \overline{T_{2}(\mu)}\\
    F & 0
    \end{bmatrix}
    \begin{bmatrix}
    I & \overline{F^{-1}(E-\mu)}\\
    0 & I
    \end{bmatrix}.
\end{equation}
\end{theorem}
\begin{proof}
Since, $F^{-1}E$ is bounded on $D(E).$ So, $\overline{F^{-1}(E- \mu)}$ is a real linear bounded operator on $\overline{D(E)}.$ When $\mathcal{A}$ is closable, then
\begin{align*}
\begin{bmatrix}
        I & -(A-\mu)F^{-1}\\
        0 & I
    \end{bmatrix}\overline{{\mathcal A}- \mu} =
 \overline{ \begin{bmatrix}
    0 & {T_{2}(\mu)}\\
    F & 0
    \end{bmatrix}
    \begin{bmatrix}
    I & {F^{-1}(E-\mu)}\\
    0 & I
    \end{bmatrix}}.
\end{align*}
Let $(0, y) \in \overline{G(T_{2}(\mu))}.$ There exists a sequence $\{x_{n}\}$ in $K$ such that $x_{n} \to 0$ and $T_{2}(\mu)x_{n} \to y,$ as $n \to \infty.$ Thus,
$$\begin{bmatrix}
    0 & {T_{2}(\mu)}\\
    F & 0
    \end{bmatrix}
    \begin{bmatrix}
    I & {F^{-1}(E-\mu)}\\
    0 & I
    \end{bmatrix}\begin{pmatrix} -F^{-1}(E -\mu)x_{n}\\ 
    x_{n}\end{pmatrix} \to \begin{pmatrix} y\\ 0\end{pmatrix},$$ as $n\to \infty.$ So, $y= 0$ which implies that $T_{2}(\mu)$ is closable.\\
    Conversely, assume that $T_{2}(\mu)$ is closable. Then,
    \begin{align*}
    {\mathcal A} - \mu \subset \begin{bmatrix}
        I & (A-\mu)F^{-1}\\
        0 & I
    \end{bmatrix}
    \begin{bmatrix}
    0 & \overline{T_{2}(\mu)}\\
    F & 0
    \end{bmatrix}
    \begin{bmatrix}
    I & \overline{F^{-1}(E-\mu)}\\
    0 & I
    \end{bmatrix}.
    \end{align*}
    This shows that \(\mathcal{A}-\mu\) is closable, and consequently, \(\mathcal{A}\) is also closable. Furthermore,
    \begin{align*}
     \begin{bmatrix}
        I & -(A-\mu)F^{-1}\\
        0 & I
    \end{bmatrix}\overline{{\mathcal A} - \mu} &=
    \overline{\begin{bmatrix}
    0 & {T_{2}(\mu)}\\
    F & 0
    \end{bmatrix}
    \begin{bmatrix}
    I & {F^{-1}(E-\mu)}\\
    0 & I
    \end{bmatrix}}\\
    &= \begin{bmatrix}
    0 & \overline{T_{2}(\mu)}\\
    F & 0
    \end{bmatrix}
    \begin{bmatrix}
    I & \overline{F^{-1}(E-\mu)}\\
    0 & I
    \end{bmatrix}.
    \end{align*}
 Therefore,  $$\overline{\mathcal A} = \mu + \begin{bmatrix}
        I & (A-\mu)F^{-1}\\
        0 & I
    \end{bmatrix}
    \begin{bmatrix}
    0 & \overline{T_{2}(\mu)}\\
    F & 0
    \end{bmatrix}
    \begin{bmatrix}
    I & \overline{F^{-1}(E-\mu)}\\
    0 & I
    \end{bmatrix}.$$  
\end{proof}
\begin{theorem}\label{thm 4.7}
Assume that $$D(B) \subset D(E),~~ B \text{ is bijective },$$ $B^{-1}A$ is bounded on $D(A)$ and $(E-\mu)B^{-1}$ is closed. Then $\mathcal A$ is closable if and only if its the quadratic complement $$T_{1}(\mu) = F - (E-\mu)B^{-1}(A-\mu)$$ is closable for some (and hence for all) $\mu \in \mathbb {C}$. Moreover, in this case, the closure $\overline{\mathcal A}$ is given by
\begin{equation}\label{equ 21}
    \overline{\mathcal A} = \mu + \begin{bmatrix}
        I & 0\\
        (E- \mu)B^{-1} & I
    \end{bmatrix}
    \begin{bmatrix}
    0 & B\\
    \overline{T_{1}(\mu)} & 0
    \end{bmatrix}
    \begin{bmatrix}
    I & 0\\
    \overline{B^{-1}(A-\mu)} & I
    \end{bmatrix}
\end{equation}
\end{theorem}
\begin{proof}
The proof proceeds analogously to that of Theorem~\ref{thm 4.6}.
\end{proof}
Let \(T\) be an antilinear operator on a Hilbert space \(H\). The \emph{point spectrum} of \(T\), denoted by \(\sigma_{p}(T)\), is defined as
\[
\sigma_{p}(T)
=
\left\{
\lambda \in \sigma(T)
:\;
T-\lambda \text{ is not injective as a real linear operator}
\right\}.
\]

The \emph{residual spectrum} of \(T\), denoted by \(\sigma_{r}(T)\), is defined by
\[
\sigma_{r}(T)
=
\left\{
\lambda \in \mathbb{C}
:\;
\begin{array}{l}
T-\lambda \text{ is injective as a real linear operator,} \\[2mm]
R(T-\lambda)=\overline{R(T-\lambda)} \neq H
\end{array}
\right\}.
\]

The \emph{continuous spectrum} of \(T\), denoted by \(\sigma_{c}(T)\), is given by
\[
\sigma_{c}(T)
=
\left\{
\lambda \in \mathbb{C}
:\;
\begin{array}{l}
T-\lambda \text{ is injective as a real linear operator,} \\[2mm]
R(T-\lambda)\neq \overline{R(T-\lambda)}
\end{array}
\right\}.
\]

Assume now that the conditions $(A -\lambda)^{-1}B$ is bounded on $D(B),$ $F(A -\lambda)^{-1}$ is closed in $H$ with $D(A) \subset D(F)$ and $S_{2}(\lambda)$ is closable, for some $\lambda \in \rho(A)$. We define the following spectra associated with the operator function $\overline{S_{2}}$ on domain $\rho(A),$ defined by $$\overline{S_{2}}(\mu) = \overline{S_{2}(\mu)}, ~ \text{ for all}~ \mu \in \rho(A).$$

\begin{enumerate}
    \item The spectrum of \(\overline{S_{2}}\):
    \[
    \sigma(\overline{S_{2}})
    =
    \left\{
    \mu \in \rho(A)
    :\;
    0 \in \sigma\!\left(\overline{S_{2}(\mu)}\right)
    \right\}.
    \]

    \item The point spectrum of \(\overline{S_{2}}\):
    \[
    \sigma_{p}(\overline{S_{2}})
    =
    \left\{
    \mu \in \rho(A)
    :\;
    0 \in \sigma_{p}\!\left(\overline{S_{2}(\mu)}\right)
    \right\}.
    \]

    \item The continuous spectrum of \(\overline{S_{2}}\):
    \[
    \sigma_{c}(\overline{S_{2}})
    =
    \left\{
    \mu \in \rho(A)
    :\;
    0 \in \sigma_{c}\!\left(\overline{S_{2}(\mu)}\right)
    \right\}.
    \]

    \item The residual spectrum of \(\overline{S_{2}}\):
    \[
    \sigma_{r}(\overline{S_{2}})
    =
    \left\{
    \mu \in \rho(A)
    :\;
    0 \in \sigma_{r}\!\left(\overline{S_{2}(\mu)}\right)
    \right\}.
    \]
\end{enumerate}

\begin{theorem}\label{thm 4.8}
Assume that $$D(A) \subset D(F), ~~ \rho(A) \neq 0,$$ and that for some $\mu \in \rho(A),$ the real linear operator $(A-\mu)^{-1}B$ is bounded on $D(B),$ while the real linear operator $F(A - \mu)^{-1}$ is closed in $H$ with the Schur complement $S_{2}(\mu) = E- \mu- F(A- \mu)^{-1}B$ is closable. Then the following statements hold:
\begin{enumerate}
\item $\sigma(\overline{\mathcal{A}}) \setminus \sigma(A) = \sigma(\overline{S_{2}}).$
\item ${\sigma_{p}}(\overline{\mathcal{A}}) \setminus \sigma(A) = {\sigma_{p}}(\overline{S_{2}}).$
\item ${\sigma_{c}}(\overline{\mathcal{A}}) \setminus \sigma(A) = {\sigma_{c}}(\overline{S_{2}}).$
\item ${\sigma_{r}}(\overline{\mathcal{A}}) \setminus \sigma(A) = {\sigma_{r}}(\overline{S_{2}}).$
\end{enumerate}
\end{theorem}
\begin{proof}
$\it(1)$ Let $\alpha \in \sigma(\overline{\mathcal{A}}) \setminus \sigma(A).$ If $0 \in \rho(\overline{S_{2}(\alpha)}),$ then the real linear operator $(\overline{S_{2}(\alpha)})^{-1}$ is bounded. So, 
\begin{align*}(\overline{\mathcal{A}}- \alpha)^{-1} &= \begin{bmatrix} I & -(A- \alpha)^{-1}\\
0 & I\end{bmatrix} \begin{bmatrix} (A - \alpha)^{-1} & 0\\
0 & (\overline{S_{2}(\alpha)})^{-1}\end{bmatrix} \begin{bmatrix} I & 0 \\ -F(A-\alpha)^{-1} & I\end{bmatrix}\end{align*}
is bounded on domain $H \times K$ which is a contradiction. Thus, $0 \in \sigma(\overline{S_{2}(\alpha)}).$ Hence, $\sigma(\overline{\mathcal{A}}) \setminus \sigma(A) \subset\sigma(\overline{S_{2}}).$\\
Conversely, assume $0 \in \sigma(\overline{S_{2}(\beta)})$ and $\beta \in \rho(A).$ Then $(A-\beta)^{-1}B$ is bounded on $D(B)$ and $F(A-\beta)^{-1}$ is bounded on  domain $H.$ If $\beta \in \rho(\overline{\mathcal{A}}),$ then 
\begin{align*} 
\begin{bmatrix} A -\beta & 0 \\0 &\overline{S_{2}(\beta)} \end{bmatrix} = \begin{bmatrix} I & 0\\ -F(A - \beta)^{-1} & I\end{bmatrix}(\overline{A}-\beta)\begin{bmatrix} I & -\overline{(A-\beta)^{-1}B} \\ 0 & I\end{bmatrix}
\end{align*}
is invertible. So, $\overline{S_{2}(\beta)}$ is invertible which is again a contradiction. Therefore, $$\sigma(\overline{\mathcal{A}}) \setminus \sigma(A) \supset \sigma(\overline{S_{2}}).$$\\
$\it(2)$ The assertion follows directly from the definitions of the point spectra of \(\overline{\mathcal{A}}\) and \(\overline{S_{2}}\).\\
$\it(3)$ Let us assume $\alpha \in \rho(A)$ and $R(\overline{S_{2}(\alpha)}) \neq \overline{R(\overline{S_{2}(\alpha)})}$ with $\overline{S_{2}(\alpha)}$ is one-one. So, $\overline{\mathcal{A}} - \alpha$ is one-one.  If $R(\overline{\mathcal{A}} - \alpha) = \overline{R(\overline{\mathcal{A}} - \alpha)},$ then $R(\overline{\mathcal{A}} - \alpha)$ is closed. It follows that
\begin{align*}
R\bigg(\begin{bmatrix} A-\alpha & 0\\ 0 & \overline{S_{2}(\alpha)}\end{bmatrix}\bigg) &= R\bigg(\begin{bmatrix} A-\alpha & 0\\ 0 & \overline{S_{2}(\alpha)}\end{bmatrix} \begin{bmatrix} I & \overline{(A-\alpha)^{-1}B}\\ 0 & I \end{bmatrix}\bigg)\\
&= R\bigg(\begin{bmatrix} I & 0 \\ -F(A-\alpha)^{-1}& I\end{bmatrix}(\overline{\mathcal{A}} - \alpha)\bigg)
\end{align*}
is closed. Hence, $R(\overline{S_{2}(\alpha)})$ is closed which is a contradiction. It shows that $${\sigma_{c}}(\overline{\mathcal{A}}) \setminus \sigma(A) \supset {\sigma_{c}}(\overline{S_{2}}).$$\\
Conversely, assume $\beta \in {\sigma_{c}}(\overline{\mathcal{A}}) \setminus \sigma(A)$ but $R(\overline{S_{2}(\beta)}) = \overline{R(\overline{S_{2}(\beta)})}.$ Then $\overline{S_{2}(\beta)}$ is one-one and
\begin{align*}
R\bigg(\begin{bmatrix} I & 0 \\ -F(A-\beta)^{-1}& I\end{bmatrix}(\overline{\mathcal{A}} - \beta)\bigg) = H \times R(\overline{S_{2}(\beta)})
\end{align*}
is closed. This shows that $R(\overline{\mathcal{A}}- \beta) $ is closed because of the invertibility of $\begin{bmatrix} I & 0\\ -F(A-\beta)^{-1}& I\end{bmatrix}$. This is a contradiction. Therefore, $${\sigma_{c}}(\overline{\mathcal{A}}) \setminus \sigma(A) \subset {\sigma_{c}}(\overline{S_{2}}).$$\\
$(\it{4})$ Let $\alpha \in \rho(A),$ $\overline{S_{2}(\alpha)}$ is one-one and $ R(\overline{S_{2}(\alpha)}) = \overline{ R(\overline{S_{2}(\alpha)})} \neq K$ but $\alpha \notin \sigma_{r}(\overline{\mathcal{A}}).$ So, $\overline{\mathcal{A}}- \alpha$ is one-one. From the condition $R(\overline{S_{2}(\alpha)}) = \overline{ R(\overline{S_{2}(\alpha)})},$ we obtain $R(\overline{\mathcal{A}} - \alpha) = \overline{R(\overline{\mathcal{A}} - \alpha)}.$ Then $R(\overline{\mathcal{A}} - \alpha) = H \times K.$ It shows that
\begin{align*}
 R \bigg(\begin{bmatrix} I & 0\\ F(A - \alpha)^{-1} & I \end{bmatrix}\begin{bmatrix} A-\alpha & 0\\ 0 & \overline{S_{2}(\alpha)}\end{bmatrix}\bigg) = H \times K.
\end{align*}
Now, let us consider $\begin{pmatrix} 0\\ k\end{pmatrix} \in H \times K.$ We get $\begin{pmatrix} u_{1}\\ u_{2}\end{pmatrix} \in H \times K$ such that 
\begin{align*}
\begin{bmatrix} I & 0\\ F(A - \alpha)^{-1} & I \end{bmatrix}\begin{bmatrix} A-\alpha & 0\\ 0 & \overline{S_{2}(\alpha)}\end{bmatrix}\begin{pmatrix} u_{1}\\ u_{2} \end{pmatrix} = \begin{pmatrix} 0 \\ k\end{pmatrix}.
\end{align*}
It is evident that $\overline{S_{2}(\alpha)}u_{2} =k.$ Thus, $R(\overline{S_{2}(\alpha})) = K$ which is a contradiction. Hence,
$${\sigma_{r}}(\overline{\mathcal{A}}) \setminus \sigma(A) \supset {\sigma_{r}}(\overline{S_{2}}).$$
Conversely, assume $\beta \in {\sigma_{r}}(\overline{\mathcal{A}}) \setminus \sigma(A).$ So, $\overline{S_{2}(\beta)}$ is also one-one. If $\beta \notin {\sigma_{r}}(\overline{S_{2}}),$ then $R(\overline{S_{2}(\beta)}) = \overline{R(\overline{S_{2}(\beta)})} = K,$ where the first identity holds because of $R(\overline{\mathcal{A}} - \beta) = \overline{R(\overline{\mathcal{A}} - \beta)}.$ Now, consider an element $\begin{pmatrix} h \\ k \end{pmatrix} \in H \times K.$ We get two elements $u \in D(A)$ and $v \in D(\overline{S_{2}(\beta)})$ such that $(A - \beta)u = h$ and $\overline{S_{2}(\beta)}v=k -Fu.$ Hence,
$$\begin{bmatrix} I & 0\\ F(A-\beta)^{-1}& I\end{bmatrix}\begin{bmatrix} A-\beta & 0 \\ 0 & \overline{S_{2}(\beta)}\end{bmatrix} \begin{pmatrix} u \\ v\end{pmatrix} = \begin{pmatrix} h \\ k\end{pmatrix}.$$ 
It shows that $R(\overline{\mathcal{A}} - \beta) = R\bigg(\begin{bmatrix} I & 0\\ F(A-\beta)^{-1}& I\end{bmatrix}\begin{bmatrix} A-\beta & 0 \\ 0 & \overline{S_{2}(\beta)}\end{bmatrix}\bigg) = H \times K$ which is a contradiction. Therefore, $${\sigma_{r}}(\overline{\mathcal{A}}) \setminus \sigma(A) \subset {\sigma_{r}}(\overline{S_{2}}).$$
\end{proof}
Assume now that the conditions $(E -\lambda)^{-1}F$ is bounded on $D(F),$ $B(E -\lambda)^{-1}$ is closed in $K$ with $D(E) \subset D(B)$ and $S_{1}(\lambda)$ is closable, for some $\lambda \in \rho(E)$. Similarly, we can define the spectra of $\overline{S_{1}}$ associated with the operator function $\overline{S_{1}}$ on domain $\rho(E),$ defined by $$\overline{S_{1}}(\mu) = \overline{S_{1}(\mu)}, ~ \text{ for all}~ \mu \in \rho(E).$$

\begin{theorem}\label{thm 4.9}
Assume that $$D(E) \subset D(B), ~~ \rho(E) \neq 0,$$ and that for some $\mu \in \rho(E),$ the real linear operator $(E-\mu)^{-1}F$ is bounded on $D(F),$ while the real linear operator $B(E- \mu)^{-1}$ is closed in $K$ with the Schur complement $S_{1}(\mu) = A- \mu- B(E- \mu)^{-1}F$ is closable. Then the following statements hold:
\begin{enumerate}
\item $\sigma(\overline{\mathcal{A}}) \setminus \sigma(E) = \sigma(\overline{S_{1}}).$
\item ${\sigma_{p}}(\overline{\mathcal{A}}) \setminus \sigma(E) = {\sigma_{p}}(\overline{S_{1}}).$
\item ${\sigma_{c}}(\overline{\mathcal{A}}) \setminus \sigma(E) = {\sigma_{c}}(\overline{S_{1}}).$
\item ${\sigma_{r}}(\overline{\mathcal{A}}) \setminus \sigma(E) = {\sigma_{r}}(\overline{S_{1}}).$
\end{enumerate}
\end{theorem}
\begin{proof}
The proof proceeds analogously to that of Theorem~\ref{thm 4.8}.
\end{proof}
Assume now that $$D(F) \subset D(A),~~ F \text{ is bijective },$$ 
$F^{-1}E$ is bounded on $D(E)$ and $(A-\lambda)F^{-1}$ is closed. The quadratic complement $$T_{2}(\lambda) =  B - (A-\lambda)F^{-1}(E-\lambda) $$  of $\mathcal{A}$ is closable for some $\lambda \in \mathbb {C}$. We define the following spectra with the operator function $\overline{T_{2}}$ on domain $\mathbb{C},$ defined by 
$$\overline{T_{2}}(\mu) = \overline{T_{2}(\mu)}, \text{ for all } \mu \in \mathbb{C}.$$
\begin{enumerate}
    \item The spectrum of \(\overline{T_{2}}\):
    \[
    \sigma(\overline{T_{2}})
    =
    \left\{
    \mu \in \mathbb{C}
    :\;
    0 \in \sigma\!\left(\overline{T_{2}(\mu)}\right)
    \right\}.
    \]

    \item The point spectrum of \(\overline{T_{2}}\):
    \[
    \sigma_{p}(\overline{T_{2}})
    =
    \left\{
    \mu \in \mathbb{C}
    :\;
    0 \in \sigma_{p}\!\left(\overline{T_{2}(\mu)}\right)
    \right\}.
    \]

    \item The continuous spectrum of \(\overline{T_{2}}\):
    \[
    \sigma_{c}(\overline{T_{2}})
    =
    \left\{
    \mu \in \mathbb{C}
    :\;
    0 \in \sigma_{c}\!\left(\overline{T_{2}(\mu)}\right)
    \right\}.
    \]

    \item The residual spectrum of \(\overline{T_{2}}\):
    \[
    \sigma_{r}(\overline{T_{2}})
    =
    \left\{
    \mu \in \mathbb{C}
    :\;
    0 \in \sigma_{r}\!\left(\overline{T_{2}(\mu)}\right)
    \right\}.
    \]
\end{enumerate}
\begin{theorem}\label{thm 4.10}
Assume that $$D(F) \subset D(A),~~ F \text{ is bijective },$$ 
$F^{-1}E$ is bounded on $D(E)$ and $(A-\mu)F^{-1}$ is closed. The quadratic complement $$T_{2}(\mu) =  B - (A-\mu)F^{-1}(E-\mu) $$ is closable for some $\mu \in \mathbb {C}$. Then the following statements hold:
\begin{enumerate}
    \item $\sigma(\overline{\mathcal{A}}) = \sigma(\overline{T_{2}}).$
    \item$\sigma_{p}(\overline{\mathcal{A}}) = \sigma_{p}(\overline{T_{2}}).$
    \item $\sigma_{c}(\overline{\mathcal{A}}) = \sigma_{c}(\overline{T_{2}}).$
    \item $\sigma_{r}(\overline{\mathcal{A}}) = \sigma_{r}(\overline{T_{2}}).$
\end{enumerate}
\end{theorem}
\begin{proof}
$\it(1)$ Let $\alpha \in \sigma(\overline{\mathcal{A}})$ and $\alpha \notin \sigma(\overline{T_{2}}).$ Then $$(\overline{\mathcal{A}}- \alpha)^{-1} = \begin{bmatrix} I & -{\overline{F^{-1}(E - \alpha)}}\\0 & I\end{bmatrix}\begin{bmatrix} 0 & F^{-1}\\(\overline{T_{2}(\alpha)})^{-1}& 0\end{bmatrix}\begin{bmatrix}I & -(A-\alpha)F^{-1}\\0 & I\end{bmatrix}$$
is bounded which is a contradiction. Hence, $\sigma(\overline{\mathcal{A}}) \subset \sigma(\overline{T_{2}}).$\\
Conversely, assume $\beta \in \sigma(\overline{T_{2}}) \cap \rho(\overline{\mathcal{A}}).$ Then $$\begin{bmatrix} 0 & \overline{T_{2}(\beta)}\\ F & 0\end{bmatrix} = 
\begin{bmatrix} I & -(A-\beta)F^{-1}\\0 & I\end{bmatrix}(\overline{\mathcal{A}} - \beta) \begin{bmatrix} I & -\overline{F^{-1}(E-\beta)} \\0 & I \end{bmatrix}$$
is bijective. It follows that $\overline{T_{2}(\beta)}$ is bijective which is again a contradiction. Therefore,  $\sigma(\overline{\mathcal{A}}) \supset \sigma(\overline{T_{2}}).$\\
$\it(2)$  The assertion follows directly from the definitions of the point spectra of \(\overline{\mathcal{A}}\) and \(\overline{T_{2}}\).\\
$\it(3)$ The assertion holds because, for some $\alpha, \beta \in \mathbb{C}$, $\overline{\mathcal{A}}-\alpha$ is one-one if and only if $\overline{T_{2}(\alpha)}$ is so, and
\[
R(\overline{\mathcal{A}}-\beta)
=
\overline{R(\overline{\mathcal{A}}-\beta)}
\]
if and only if
\[
R(\overline{T_{2}(\beta)})
=
\overline{R(\overline{T_{2}(\beta)})}.
\]
$\it(4)$ Let $\alpha \in \sigma_{r}(\overline{\mathcal{A}})$ but $\alpha \notin \sigma_{r}(\overline{T_{2}}).$ Then $R(\overline{T_{2}(\alpha)})
=
\overline{R(\overline{T_{2}(\alpha)})} = H.$ Consider an element $\begin{pmatrix}h \\ k\end{pmatrix} \in H \times K.$ Then, we get $v \in H$ and $w \in K$ such that $Fv = k$ and $\overline{T_{2}(\alpha)}w = h -(A-\alpha)v.$ So,
$$\begin{bmatrix} I & (A- \alpha)F^{-1}\\ 0 & I\end{bmatrix} \begin{bmatrix}0 & \overline{T_{2}(\alpha)}\\F & 0\end{bmatrix} \begin{pmatrix}v \\ w\end{pmatrix} = \begin{pmatrix}h \\k\end{pmatrix}.$$ It follows that $$R(\overline{\mathcal{A}}- \alpha) = R\bigg(\begin{bmatrix} I & (A- \alpha)F^{-1}\\ 0 & I\end{bmatrix} \begin{bmatrix}0 & \overline{T_{2}(\alpha)}\\F & 0\end{bmatrix} \bigg) = H \times K.$$ This is a contradiction. Hence,  $\sigma_{r}(\overline{\mathcal{A}}) \subset \sigma_{r}(\overline{T_{2}}).$\\
Conversely, assume $\beta \in \sigma_{r}(\overline{T_{2}})$ but $\beta \notin \sigma_{r}(\overline{\mathcal{A}}).$ Then $R(\overline{\mathcal{A}}-\beta)
=
\overline{R(\overline{\mathcal{A}}-\beta)}
= H \times K.$ For arbitrary $h \in H,$ there exists $\begin{pmatrix} u_{1} \\ u_{2}\end{pmatrix}$ such that $$\begin{bmatrix} I & (A- \beta)F^{-1}\\ 0 & I\end{bmatrix} \begin{bmatrix}0 & \overline{T_{2}(\beta)}\\F & 0\end{bmatrix} \begin{pmatrix} u_{1} \\ u_{2}\end{pmatrix} = \begin{pmatrix} h \\ 0\end{pmatrix}.$$ This shows that $\overline{T_{2}(\beta)}u_{2} = h,$ which is a contradiction. Therefore,  $\sigma_{r}(\overline{\mathcal{A}}) \supset \sigma_{r}(\overline{T_{2}}).$
\end{proof}
\begin{theorem}\label{thm 4.11}
Assume that $$D(B) \subset D(E),~~ B \text{ is bijective },$$ 
$B^{-1}A$ is bounded on $D(A)$ and $(E-\mu)B^{-1}$ is closed. The quadratic complement $$T_{1}(\mu) =  F - (E-\mu)B^{-1}(A-\mu) $$ is closable for some $\mu \in \mathbb {C}$. Then the following statements hold:
\begin{enumerate}
    \item $\sigma(\overline{\mathcal{A}}) = \sigma(\overline{T_{1}}).$
    \item$\sigma_{p}(\overline{\mathcal{A}}) = \sigma_{p}(\overline{T_{1}}).$
    \item $\sigma_{c}(\overline{\mathcal{A}}) = \sigma_{c}(\overline{T_{1}}).$
    \item $\sigma_{r}(\overline{\mathcal{A}}) = \sigma_{r}(\overline{T_{1}}).$
\end{enumerate}
\end{theorem}
\begin{proof}
The proof proceeds analogously to that of Theorem~\ref{thm 4.10}.
\end{proof}

\section*{Author Contributions} 
The author would like to thank Prof. Vladimir M\"uller for fruitful discussions and valuable suggestions during the preparation of this manuscript, as well as for his supervision.
\section*{Funding}
The author gratefully acknowledge support from GA~\v{C}R(Grant No.~25-15444K) and RVO:67985840.
\section*{Data Availability}
Data sharing is not applicable to this article as no data sets were generated or analyzed during the current study.

\section*{Declarations}
  The author confirms that there are no conflicts of interest regarding the publication of this manuscript.


\begin{thebibliography}{10}



\bibitem{majumdar2024hyers}
           Arup Majumdar, P Sam Johnson, and Ram N Mohapatra.\newblock{Hyers-Ulam Stability of Unbounded Closable Operators in Hilbert Spaces}. \newblock{\em Mathematische Nachrichten}, 297(10):3887-3903, 2024.

  

\bibitem{MR007}
   Arup Majumdar, P. Sam Johnson.
   \newblock{The Moore-Penrose inverses of unbounded closable operators and the direct sum of closed operators in Hilbert spaces}, \newblock{\em Linear and Multilinear Algebra}, 73(8):1668–1684, 2025.  

\bibitem{MR0396607}
  Arup Majumdar, P. Sam Johnson, Ram N. Mohapatra.
   \newblock{On the generalized Cauchy dual of closed operators in Hilbert
spaces}. \newblock{\em Acta Scientiarum Mathematicarum}, 2025. (To appear)

\bibitem{MRARUP}
Arup Majumdar, P. Sam Johnson, and Ram N Mohapatra.
\newblock{Characterizations of closed EP operators on Hilbert spaces}.
\newblock{\em Operators and Matrices}, 2026. (To appear)

\bibitem{MR2463978}
	Christiane Tretter.
	\newblock {\em Spectral theory of block operator matrices and applications}.
	\newblock Imperial College Press, London, 2008.

\bibitem{Vmuller}
 Damian Kołaczek, Vladimir Müller. 
 \newblock{Numerical ranges of antilinear operators,} \newblock {\em Integral Equations and Operator Theory} 96(2): 17, 2024.

\bibitem{EKo}
Eungil Ko, Ji Eun Lee, and Mee-Jung Lee. 
\newblock{On properties of C-normal operators,} \newblock{\em Banach Journal of Mathematical Analysis,} 15(4):65, 2021.


\bibitem{Fherbut}
F. Herbut, M. Vujičić. 
\newblock{Basic algebra of antilinear operators and some applications.I,} \newblock{Journal of Mathematical Physics,} 8(6): 1345-1354, 1967.

\bibitem{IanK}
Ian Knowles. 
\newblock{On $J$-selfadjoint extensions of $J$-symmetric operators,} \newblock{\em Proceedings of the American Mathematical Society,} 79(1):42-44, 1980.

\bibitem{IonG}
Ion Geru. 
\newblock{Time-Reversal Symmetry Seven Time-Reversal Operators for Spin Containing Systems,} \newblock{\em Springer Tracts in Modern Physics,} pp.1-350, 2018.

\bibitem{stochel1989normal}
  Jan Stochel, F. H. Szafraniec.
  \newblock{On normal extensions of unbounded operators. II},
  \newblock{\em Acta Sci. Math.(Szeged)}, 53(1-2):153--177, 1989.

\bibitem{stochel3}
Jan Stochel, F. H. Szafraniec. 
\newblock {On normal extensions of unbounded operators. III. Spectral properties,} \newblock{Publications of the Research Institute for Mathematical Sciences,} 25(1):105-139, 1989.

\bibitem{jawvan}
J. van Casteren, Seymour Goldberg.
\newblock{The conjugate of the product of operators,} \newblock{\em Studia Mathematica,} 125-130, 1970.
 
\bibitem{MR2953553}
	Konrad Schm\"{u}dgen.
	\newblock {\em Unbounded self-adjoint operators on {H}ilbert space}, volume 265
	of {\em Graduate Texts in Mathematics}.
	\newblock Springer, Dordrecht, 2012.

\bibitem{martin}
Martin Schechter.
\newblock{The conjugate of a product of operators,} \newblock{\em Journal of Functional Analysis,} 6:26-28, 1970.


\bibitem{ptak2019c}
  Marek Ptak,  Katarzyna  Simik, and Anna Wicher.
  \newblock{$ C $--normal operators}, \newblock{\em Electronic Journal of Linear Algebra,} 36:67-79, 2020.
        
 \bibitem{MR0203464}
	 R. G. Douglas.
	\newblock {On majorization, factorization, and range inclusion of operators on Hilbert space}.
	\newblock {\em Proc. Amer. Math. Soc.,} 17:413--415, 1966.    

\bibitem{bargmann}
Valentine Bargmann. 
\newblock{Note on Wigner's theorem on symmetry operations,} \newblock {Journal of Mathematical Physics,} 5(7):(862-868), 1964.

\bibitem{Arlinskii}
Y. Arlinskii,  K. Schmüdgen.
\newblock{On C-Symmetric and C-Self-adjoint Unbounded Operators on Hilbert Space,} \newblock{\em Integral Equations and Operator Theory}, 98(2):14, 2026.
\end{thebibliography}
\end{document}